\documentclass[12pt,leqno]{amsart}
\usepackage[dvips]{graphics}
\usepackage{amssymb}
\usepackage{verbatim}
\usepackage{hyperref}
\usepackage{color}

\numberwithin{equation}{section}
\newdimen\vintkern\vintkern12pt
\def\vint{-\kern-\vintkern\int}

\newtheorem{thm}{Theorem}[section]

\newtheorem{lem}[thm]{Lemma}
\newtheorem{cor}[thm]{Corollary}
\newtheorem{prop}[thm]{Proposition}

\newtheorem{quest}[thm]{Question}

\newcommand{\tref}[1]{Theorem~\ref{#1}}
\newcommand{\cref}[1]{Corollary~\ref{#1}}

\newcommand{\R}{\mathbb{R}}

\begin{document}
	\pagebreak
	
	
	\title{Smoothness of Submetries}
	
\thanks{A. L. was partially supported by the DFG grant   SFB TRR 191.}

	\author{Alexander Lytchak}
	
	\email{alexander.lytchak@kit.edu}

\author{Burkhard Wilking}

\email{wilking@uni-muenster.de}	


	\begin{abstract}
We prove that a Riemannian submersion between smooth, compact, non-negatively curved
Riemannian manifolds has to be smooth, resolving a conjecture by Berestovskii--Guijarro.  We show that without any curvature assumption, the smoothness of the base is implied by the smoothness of the total space.  Results are proven in the much more general setting of submetries. These are  metric generalizations of Riemannian submersions and isometric actions, which have recently appeared in different areas of geometry.
%
	\end{abstract}
	
	\maketitle

	\renewcommand{\theequation}{\arabic{section}.\arabic{equation}}
	\pagenumbering{arabic}

	\section{Introduction}
\subsection{Smoothness of submersions}	
 A  \emph{submetry} is a map $P:X\to Y$  between metric spaces  which 
sends balls  in $X$ onto  balls  of the same radius  in $Y$.
Submetries as  metric generalization of Riemannian submersions and quotient maps for isometric group actions were introduced  in \cite{Berest2}. Valeryi  Berestovksii and Luis Guijarro  verified that a submetry between smooth complete Riemannian manifolds  has to be a $\mathcal C^{1,1}$-Riemannian submersion, but  does not need to be $\mathcal C^2$,  \cite{Berest}.  They conjectured that non-smoothness cannot occur in non-negative curvature, a statement motivated by the 
conjectural smoothness of the Sharafudtinov retraction \cite{Per-soul}, verified  later in \cite{Wilking}. The second statement of the following theorem, one of the main results of this paper, confirms this conjecture in the compact case.

	\begin{thm} \label{thm-sm1}
	Let $M,N$ be $\mathcal C^1$ Riemannian manifolds and $P:M\to N$ be a surjective  $\mathcal C^1$ Riemannian submersion. If  
	the metric on $M$ is smooth then the Riemannian manifold $N$ is smooth as well.
	
If, in addition, $M$ is compact and has non-negative sectional curvature, then $P$ is a smooth map. 
\end{thm}

The first statement of  
Theorem \ref{thm-sm1} answers a question formulated by Frank Lempert  \cite[p. 372]{Lem}, a paper in which this first statement was proved for real-analytic metrics by means of complex analysis.


As we are going to show, Theorem \ref{thm-sm1} essentially  remains  valid for submetries  $P:M\to N$ without any a priori assumption on 
$N$.

\subsection{Smoothness of submetries, the base}


Submetries $P:X\to Y$  are in one-to-one correspondence with equidistant decompositions
of the \emph{total space} $X$. The correspondence assigns to $P$ the decomposition of $X$ by fibers of $P$ \cite[Section 2.2]{KL}.  For an equidistant decomposition of $X$, the \emph{base space} $Y$ appears as an auxiliary object and  it is natural  to understand its properties only in terms of $X$.   Motivated by recent appearances
of submetries in many unrelated settings \cite{s24rigid}, \cite{Stadler}, \cite{StadlerII}, \cite{Lyt-buildings}, \cite{Rade},  \cite{Mendes-Rad}, \cite{Mendes-Rad2}, \cite{MR-Weyl},
  \cite{Grovesubmet},   \cite{GW}, \cite{Mondino1}, 
\cite{LPZ}, \cite{Ber-finite}, \cite{Ber-Nik-hom},  \cite{Lange},
a systematic investigation of submetries $P:M\to Y$ with  $M$ a  Riemannian
manifold was initiated in \cite{KL}.

The space $Y$ has to be a length metric space of curvature locally bounded from below \cite[Proposition 3.1]{KL}.  There is a canonical stratification 
$Y= \cup _{e=0} ^m Y^e$, where $Y^e$ is a locally convex subset, locally isometric to an $e$-dimensional  Riemannian manifold with Lipschitz continuous metric tensor, \cite[Theorem 1.6 Theorem 11.1]{KL}. 
The largest stratum $Y^m$, the set of \emph{regular points} of $Y$, is open and convex in $Y$. The restriction $P:P^{-1} (Y^m)\to Y^m$ is a $\mathcal C^{1,1}$-Riemannian submersion.

 Our main results concern smoothness of the base and of the fibers of $P$.  The next theorem generalizes the first part of Theorem \ref{thm-sm1}.

\begin{thm} \label{thm: new3}
Let $M$ be a smooth Riemannian manifold, let $P:M\to Y$ be a submetry and let $Y=\cup _{e=0} ^m Y^e$ be the canonical stratification of $Y$. Then all  strata $Y^e$ are   smooth Riemannian manifolds.
\end{thm}

The smoothness of the stratum  $Y^e$ is not uniform near points of $\overline{Y^e} \setminus Y^e$. For instance,  the curvature in the maximal stratum $Y^m$
usually explodes near points in strata of larger codimension \cite[Theorem 1.1]{Thorb}. If the submetry is given by an isometric group action, or, more generally, by the leaves of a smooth \emph{singular Riemannian foliation}, see \cite{Molino}, \cite{Alexandrino-survey}, \cite{Thorb} and Section \ref{sec: srf} below, then this (non)-explosion is characterized by infinitesimal data in \cite[Theorem 1.4]{Thorb}.  In these cases, $Y$ is a smooth Riemannian orbifold up to strata of codimension $\geq 3$ \cite[Proposition 3.1]{Thorb}.  For  submetries,
 we prove:

\begin{thm} \label{thm: codim1}
Let $M$ be a smooth Riemannian manifold and $P:M\to Y$ be a submetry with $\dim (Y)=m$.  Then any point $y\in Y^{m-1}$  has a  neighborhood in $Y$  isometric to a smooth Riemannian manifold with totally geodesic boundary. 
\end{thm}

As a consequence \emph{almost every quasi-geodesic} in $Y$ is contained in  
the open, convex  subset of $Y$ isometric to a smooth Riemannian manifold with totally geodesic boundary, therefore, it can be be controlled by classical means of Riemannian geometry, Lemma \ref{lem: no-conj}.

\subsection{Smoothness of submetries, the fibers} Let again $P:M\to Y$ be a submetry, where $M$ is a smooth Riemannian manifold.
 \emph{Most} fibers of $P$ are $\mathcal C^{1,1}$-submanifolds of $M$, but some fibers may be arbitrary subsets of positive reach in $M$ with rather wild singularities  \cite[Theorem 1.8, Examples 6.6-6.10, Proposition 6.3]{KL}.  It may appear somewhat surprising that the possibly wild singularities of the fibers provide no restrictions on the smoothness of the base, as stated in Theorem \ref{thm: new3}.

On the other hand, one cannot hope for higher regularity of the fibers   without
some  priori assumption, even if the total space is the round sphere or the Euclidean space:
 The distance function to a ray $P:\mathbb R^2\to [0,\infty)$ is a  submetry and the fibers of $P$ are not $\mathcal C^2$.

A submetry $P:M\to Y$ is called \emph{transnormal} (or \emph{manifold submetry}) if
all fibers of $P$ are topological manifolds without boundary.   This property   is equivalent to the statement that geodesics starting normally to a fiber are normal to all fibers they intersect. Transormality is a \emph{stable} property, \cite[Corollary 12.6]{KL};  major classes of submetries including Riemannian submersions, isometric group actions and singular Riemannian foliations with closed leaves are transnormal. Moreover, submetries appearing in connections with the rank rigidity \cite{Lyt-buildings}, \cite{Stadler}, \cite{StadlerII} and 
submetries appearing in connections with  Laplacian algebras \cite{Mendes-Rad}, \cite{Mendes-Rad2}, \cite{MR-Weyl} are transnormal.

Transnormality turns out to be   necessary for  smoothness of the fibers in the regular part; this answers a question  by Marco Radeschi:

\begin{thm} \label{thm:last}
Let $M$ be a smooth Riemannian manifold and $P:M\to Y$ be a submetry. If, for some $k \geq 2$,  all regular fibers are $\mathcal C^k$ submanifolds  then $P$ is transnormal and all fibers of $P$ are $\mathcal C^k$ submanifolds.
\end{thm}

\subsection{Smoothness of submetries, non-negative curvature}
Finally, we turn to non-negative curvature and formulate our main results.

\begin{thm} \label{thm: maincompact}
Let $M$  be complete and  of positive sectional curvature or
  a non-negatively curved  symmetric space.  Let $P:M\to Y$ be a transnormal submetry. Then all fibers of $P$ are smooth.
\end{thm}

Theorem \ref{thm: maincompact} should be valid  for all complete $M$ of non-negative curvature.  In this generality  we can verify  smoothness almost everywhere:

\begin{thm} \label{thm: main}
 Let  $M$ be a complete Riemannian manifold of non-negative sectional curvature and let $P:M\to Y$ be a transnormal submetry. Then all fibers of $P$ are $\mathcal C^2$. There is an open, dense subset $U$ of $M$ such that the intersection of $U$ with any fiber of $P$ is a smooth submanifold of $M$.
\end{thm}

In order to fully  extend Theorem \ref{thm: maincompact} to non-negative curvature, a better understanding of the \emph{dual foliation}   would be needed,  essentially a resolution of   \cite[Conjecture]{Wilking}.  The resolution of this conjecture for  submersions \cite[Theorem 3]{Wilking} and  symmetric spaces \cite{Llohann} implies  the remaining parts of Theorem \ref{thm-sm1} and Theorem \ref{thm: maincompact}.

\subsection{Additional statements and questions}  

\subsubsection{Locality}   The statements and the proofs  of Theorems \ref{thm: new3}, \ref{thm: codim1}, \ref{thm:last} are local. They do not require completeness of $M$ and are valid for \emph{local submetries}:  a generalization of submetries which covers Riemannian submersions between non-complete manifolds and restrictions of submetries to open subsets, see Section \ref{subsec: loc}, Theorems \ref{thm-sm3}, \ref{thm-sm4}.

\subsubsection{O'Neill formula}
While the classical proof of the O'Neill formula for Riemannian submersions $P$ does not work if  $P$ is not  $\mathcal C^3$, it is still possible to verify the almost everywhere validity of  the  formula by different means.  We   state  the result in the  following  form:
  
\begin{prop} \label{prop: oneil}
Let  $M,Y$ be  smooth Riemannian manifolds and let $P:M\to Y$ be a $\mathcal C^{1,1}$ Riemannian submersion.  Then, for any $x\in M$ and any two-dimensional plane $E$ in the horizontal space $H^x\subset T_xM$, the sectional curvatures of $E$ in $M$ and of $E':=D_xP(E)$ in $Y$ satisfy
\begin{equation}  \label{eq: veryfirst}
\kappa (E) \leq \kappa (E')\;.
\end{equation}
Equality in   \eqref{eq: veryfirst} holds  for all  $x\in M$ and  all horizontal planes $E \subset H^x$ if and only if any $x\in M$ is contained in a totally geodesic submanifold $N^x$ of $M$, which is sent by $P$ isometrically onto an  open subset of $Y$.
\end{prop}

For any \emph{submetry} $P:M\to Y$, the above Proposition \ref{prop: oneil} applies to the open dense subset of regular fibers $P:M_{reg} \to Y_{reg}$.     It should be possible to verify  in case of equality in \eqref{eq: veryfirst} on all of $M_{reg}$, that  
the quotient space $Y$ is a smooth Riemannian orbifold.  If in addition, the  submetry is transnormal, it should be possible to recover the whole structure of a polar foliation (also known as singular Riemannian foliation with sections), \cite{Thorb-survey}, \cite{Thorb-survey2}, \cite{Alex-section}, \cite{Thorbergsson}, \cite{Chi}, \cite{Terng}, \cite{Ly-polar}.    We leave this to the interested reader.

\subsubsection{Equifocality}  \label{subsec: equifocality}
 Let  $M$ be a smooth, \emph{complete} Riemannian manifold and let $P:M\to Y$ be a transnormal submetry. 
 For any regular fiber $L$ of $P$, any normal vector $h\in T^{\perp} _p L$ to $L$  extends uniquely to a \emph{basic normal field}  $\hat h_q \in T^{\perp}  _q L$ along $L$, such that the geodesics in the directions of $\hat h_q$ are sent by $P$ to the same quasi-geodesic in $Y$, independent of $q$, 
\cite[Proposition 12.7]{KL}. 
  It is an important observation, that  the focal times of $L$ along the geodesic $\gamma _q (t):=\exp _q (t\cdot \hat h_q) $ do not depend on $q$, Proposition \ref{jac: family}, a statement  known under the name of \emph{equifocality}
in the theory of singular Riemannian foliations, \cite{AT1}.

  Moreover, any $L$-Jacobi field $J^p$ along $\gamma _p$, which vanishes  at $t$, extends to a family of $L$-Jacobi fields $J^q$  along $\gamma _q$ vanishing at the time $t$,   for all $q\in L$ close to $p$, Theorem \ref{jac: family}.

\subsubsection{Basic mean curvature}
We say that a submetry $P:M\to Y$ has \emph{basic mean curvature} if for any regular fiber $L$ of $P$, the mean curvature vector is a basic normal vector field. This property plays an important role in the theory of (singular) Riemannian foliations,  cf. \cite{MR-Weyl}.

For a submetry with basic mean curvature, the  mean curvature vector of any regular fiber is  as smooth as the differential of $P$.
  Using  elliptic regularity and a  boot-strap argument we show:

\begin{prop} \label{thm:  smoothbasic}
If  a submetry $P:M\to Y$ has basic mean curvature then all fibers of $P$ are smooth. 
\end{prop}

In the Euclidean space, the eigenvalues of  the  second fundamental form of a submanifold are directly linked with the distances to focal points, \cite{PalT}. Hence, equifocality implies that all transnormal submetries $P:\mathbb R^n\to Y$
have basic mean curvature.    Thus, Proposition \ref{thm: smoothbasic} implies
Theorem \ref{thm: main} for $M=\mathbb R^n$.

As has been recently verified in \cite{LMR-preprint}, all transnormal submetries on non-negatively curved normal homogeneous Riemannian manifolds have basic mean curvature. Thus, Proposition \ref{thm: smoothbasic} applies  and provides an alternative proof of Theorem \ref{thm: maincompact} in this case.  Indeed, similarly to \cite{Mendes-Rad}, in this special case one can go beyond smoothness and deduce algebraicity of the fibers.

 \subsubsection{Dual leaves and extension of smoothness} \label{subsec: extsm}
For a transnormal submetry $P:M\to Y$ and a point $x\in M$, the \emph{dual leaf} of $x$ is the set of points in $M$ connected with $x$ by  horizontal piecewise geodesics, cf. \cite{Wilking}. If all fibers of $P$ are $\mathcal C^2$, the dual leaves turn out to be leaves of  a $\mathcal C^1$-foliation on $M$, Proposition \ref{prop: dualck}.

Smoothness propagates along dual leaves.   The proof  of this statement is not  difficult in the regular part of $M$ and over codimension one-singularities of $Y$,
 Proposition  \ref{prop: extreg1fib}.  The general case of this statement  is considerably more technical, Theorem \ref{thm: dualleaf}.  

In this slightly less regular situation than  in \cite{Wilking}, most results of \cite{Wilking} remain  valid without changes.  In particular, if $M$ is complete, non-negatively curved and all dual leaves are complete then the dual foliation is  a singular Riemannian foliation \cite[Theorem 2]{Wilking}.    In this case, all leaves of the submetry $P$ are smooth, Theorem
\ref{thm: dualcomplete+}.

\subsubsection{Open questions}  We would like to mention the following problems on the structure of submetries. 
The first question is a  generalization of the  Molino conjecture \cite{Molino}, \cite{Molino1}  resolved in \cite{ARMolino}.

\begin{quest}
Is any  transnormal submetry  with (sufficiently smooth) connected leaves  a (suffciently smooth) singular foliation?
\end{quest}
	
While the $\mathcal C^2$-case of the next result is covered by  Proposition \ref{prop: dualck}, the general $\mathcal C^{1,1}$-case would require an extension of  the classical Stefan--Sussmann theorem, \cite{Stefan}:

\begin{quest}
Are the dual leaves of a transnormal submetry leaves of a singular foliation?
\end{quest}

The presence of submetries  requires some rigidity if the total manifold is compact and "strange" examples of submetries are not easy to construct. We do not have an example showing that \emph{any} curvature assumption is indeed  needed in the second statement of Theorem \ref{thm-sm1}.

\begin{quest}
Are there examples of  non-smooth Riemannian submersions between \emph{compact} smooth Riemannian manifolds?
\end{quest}

We expect that most of the theory of polar foliations extend to the case of (local) transnormal submetries satisfying equality in \eqref{eq: veryfirst}. Similarly, we expect that the infinitesimal polarity is equivalent to the non-explosion of the curvature as in \cite{Thorb}:

\begin{quest}
Let $P:M\to Y$ be  a transnormal submetry,  $x\in M$.  Can the property that a neighborhood of $P(x)$ in $Y$ is isometric to a Riemannian orbifold  be read off  the infinitesimal submetry $D_xP$?
\end{quest}

While we use the base space, most results about  the submetry can be formulated in terms of the total space only, moreover, most of our results are local. This gives rise to the following natural:
\begin{quest}
How much of the theory of (local) submetries can be extended to transnormal decompositions of a smooth Riemannian manifold $M$ in $\mathcal C^1$, possibly non-closed leaves.
\end{quest}

\subsection{Plan of the paper and sketches of the proofs}

\subsubsection{Preparations} In Sections \ref{sec: not}-\ref{sec: basicnormal}  we set up notations and collect basic tools. We recall the transnormal Jacobi equation from \cite{Wilking} and the   relation between the smoothness of a submanifold $L$ of a smooth Riemannian manifold,  of the  second fundamental form  of $L$ and of $L$-Jacobi  fields.  Finally, we recall structural basics of submetries as developped in \cite{KL} and properties of  basic  normal fields, as used in the theory of Riemannian foliations.

\subsubsection{Quotient space}   Sections \ref{sec: good}-\ref{sec: on} are devoted to smoothness properties of the quotient space.
This  smoothness  stated in Theorems \ref{thm: new3},  \ref{thm: codim1}   is obtained along the following lines. For simplicity, we assume that $P$ is a Riemannian submersion as in the first part of Theorem \ref{thm-sm1}.  Given $y\in Y$ and $x\in L:=P^{-1} (y)$,   the composition $P\circ \exp _x$ is a $\mathcal C^{1,1}$-diffeomorphism between a  neighborhood $O_x$ of $0_x$ in the normal space $T^{\perp} _x L$  and a neighborhood of $y$ in $Y$.  We pull-back the metric from $Y$ to $O_x$ and express this pull-backed metric in terms of the curvature tensor of $M$ along $\exp _x (O_x)$ and the second-fundamental form of $L$ at the point  $x$ only.  Thus,  the non-smoothness of $L$ does not play a role.

The above proof applies without changes to all  strata.  The proof that a neighborhood of a point on a stratum of codimension one in the quotient $Y$ is isometric to a smooth Riemannian orbifold needs to overcome additional difficulties.  A major technical detail is that the local structure of a submetry  cannot be directly linked to the much better understood  infinitesimal structure, unlike in the case of isometric group actions or singular Riemannian foliations, \cite{Thorb}, \cite{Mendes-Rad-slice}, \cite[Questions 1.14,1.15]{KL}.  The main additional point in the proof of Theorem \ref{thm: codim1} is the existence of a slice, which we only verify at points projected to strata of codimension one.

Once Theorem \ref{thm: codim1} is verified, we infer that \emph{most} horizontal geodesics in the total space of a transnormal submetry $P$ are projected by $P$ to the smooth  Riemannian orbifold part of the quotient. Moreover, it turns out that along such \emph{typical geodesics} the transversal Jacobi equation coincides with the Jacobi equation on the quotient space, upon a canonical identification.  This implies the validity of the O'Neill formula, Proposition \ref{prop: oneil}, and simplifies  later considerations related to the higher smoothness of fibers of $P$.



 


\subsubsection{Smoothness of fibers and focal points}
In Sections \ref{sec: nontrans}-\ref{sec: euclid}, we discuss smoothness properties of fibers without direct relation to non-negative curvature.  In Section \ref{sec: nontrans} we verify that a non-transnormal submetry must have regular fibers which are not $\mathcal C^2$. 
The simple idea is that non-transnormality of $P:M\to Y$ implies the existence of a point $x\in M$ such that the differential $D_xP$ is a product  of a linear projection and the distance 
function to a ray in some Euclidean space of dimension $\geq 2$.  In this infinitesimal submetry, regular fibers close to the origin of the ray are $\mathcal C^{1,1}$ but their second-fundamental forms have  huge jumps in the direction of the ray. Somewhat tedious    
but straightforward estimates are needed to extend this infinitesimal non-smoothness to a neighborhood of $x$ in $M$.
 
In Section \ref{sec: cont}, we show that it suffices to control the smoothness of regular fibers and that smoothness propagates along \emph{typical} horizontal geodesics, thus finishing the proof of Theorem \ref{thm:last}.

In Section \ref{sec: focal}, we verify the equifocality  of transnormal submetries  explained in Section \ref{subsec: equifocality} above.  We prove that along \emph{typical} geodesics projected by $P$ to a fixed  quasi-geodesic in the quotient, focalizing Jacobi fields depend as smoothly on the base point as the starting directions of the geodesics.  This somewhat technical statement is the central point of the whole paper.  It will allow  us  to apply a boot-strap argument, once regular fibers have sufficiently many focal points.

The effect of equifocality on smoothness is exemplified in the case of the Euclidean total space in Section \ref{sec: euclid}.  In this case,
 the eigenvalues of  second fundamental form of a submanifold are directly linked with the distances to focal points, \cite{PalT}. Hence, equifocality implies that all transnormal submetries $P:\mathbb R^n\to Y$
have basic mean curvature.    Thus, Theorem \ref{thm: smoothbasic} implies that any such $P$ is smooth.

In Section \ref{sec: c2}-\ref{sec: applic} we finally add the assumption that the total space is complete and  non-negatively curved.  Then any space of $L$-Jacobi fields is generated by focalizing Jacobi fields and parallel Jacobi fields, as proved in \cite{Wilking}.
In Section \ref{sec: c2}, we  use this description together with stability of focal indices  and deduce from Section \ref{sec: focal}, that  the spaces  of $L$-Jacobi fields along geodesics defined by a basic normal field depend continuously on the starting point.  This implies that  
all regular fibers, and hence all fibers are $\mathcal C^2$.

Assume now that  no parallel normal Jacobi fields exist in $M$, thus all spaces of $L$-Jacobi fields are generated by focalizing fields only. In this case the smoothness is obtained along the following lines:

On a $\mathcal C^k$ submanifold $L$, the normal bundle is $\mathcal C^{k-1}$ and the second-fundamental form $\mathcal C^{k-2}$.   In order to upgrade the smoothness of the submanifold to $\mathcal C^{k+1}$, it suffices to show that the second fundamental form is $\mathcal C^{k-1}$.  This upgrade is achieved, once we can show that $L$-Jacobi fields
in direction of (sufficiently many)  normal  fields are as smooth as the normal fields (and not worse), Section \ref{sec: smooth}.  However, this  controlled dependence of focalizing Jacobi-fields  is exactly the consequence of equifocality discussed  in Section \ref{sec: focal}.
This allows us a step-by-step upgrade to smoothness.

This argument can be extended to general non-negative curvature, if spaces of parallel Jacobi fields vary continuously.
In general,  they vary only semi-continuously, as the dimension of the space of parallel Jacobi fields can  go up in the limit.  However, this does not happen on an open and dense. Using this we show in Section \ref{sec: rank} that  the submetry is smooth on an open and dense subset.  

In Section \ref{sec: dual}, we  show that smoothness propagates along all horizontal geodesics and that the sets of points connected by piecewise horizontal geodesics  constitute a singular foliation of the total manifold $M$.   Our proof of this result, a generalization of the corresponding statement in \cite{Wilking}, does not require any completeness  or curvature assumptions but uses  the $\mathcal C^2$-property, verified previously in non-negative curvature.

In the final Section \ref{sec: applic}, we show that the almost everywhere smoothness can be upgraded to smoothness, once this dual foliation is shown to be a singular Riemannian foliation.  Results  in \cite{Wilking} and \cite{Llohann} imply that the dual foliation is indeed Riemannian if $M$ is a symmetric space  or if $M$ is compact and $P$ is a Riemannian submersion.

\subsection{Aknowledgements}  The authors would like to thank Valeriy Berestovskii, Ricardo Mendes, Marco Radeschi, Stephan Stadler for helpful comments.

\newpage
\centerline {\bf  I.  Preliminaries}

\section{Notation} \label{sec: not}
\subsection{Riemannian} $M$ will always denote a smooth Riemannian manifold, possibly non-complete. The dimension of $M$ will always be denoted by $n$.   If non-negative or positive curvature is mentioned, it always refers to the sectional curvature.

For natural $k$ and $0<\alpha \leq 1$, 
we call a map $\mathcal C^{k,\alpha}$ if 
the map is $\mathcal C^k$ and the $k$-th derivatives are \emph{locally} $\alpha$-H\"older. Similarly, we define $\mathcal C^{k, \alpha}$-submanifolds. As usual, we set 
 $\mathcal C^{k, 0}:=\mathcal C^k$.

For a vector $h$ in the tangent bundle $TM$ we denote by $\gamma ^h$ the geodesic with initial direction $(\gamma^h)'(0)=h$.

For a $\mathcal C^2$ submanifold $L$ of $M$ and a point $x\in L$ we denote by 
$$\mathrm{II}=\mathrm{II}_L=\mathrm{II}_x=\mathrm{II}_{L,x}:T_xM\times T_xM\to T^{\perp}_xM$$
 the second fundamental form of $L$ at $x$. For  $v,w\in T_xL$ it is given by 
$$\mathrm{II} (v,w)= (\nabla _V W )^{\perp} (x)\;,$$
where $V$ and $W$ are extensions of $v$ and $w$ to Lipschitz continuous tangent fields along $L$, which are differentiable at $x$ and where $^{\perp}$ means the orthogonal projection onto the normal space $T_x^{\perp} L$.   

This also defines a second fundamental form   for any $\mathcal C^{1,1}$ submanifold $L$ of $M$ at all points $x\in L$, at which $L$ is twice differentiable.

For a normal vector $h\in T_x^{\perp}L$  we consider the \emph{shape operator $S^h:T_xL\to T_xL$  in the direction $h$} 
  and the \emph{second fundamental form $\mathrm{II}^h :T_xL\times T_xL\to \mathbb R$  in the direction $h$}:
\begin{equation} \label{eq: defsh} 
\langle S^h (v),w \rangle: =\langle \mathrm{II} (v,w), h \rangle =: \mathrm{II}^h (v,w)\;.
\end{equation}

\subsection{Metric}
By $d$ we denote the distance in metric spaces.
For a subset $A$ of a metric space $X$ we denote by $d_A:X\to \R$ the distance function to  $A$
and by $B_r(A)$ the open $r$-neighborhood around $A$ in  $X$.
By $\ell (\gamma)$ we denote the length of a curve $\gamma$ in a metric space $X$.

For \emph{geodesics}, we adopt the convention of Riemannian geometry:  
A \emph{geodesic} will denote a  locally length-minimizing curve parametrized proportionally to arclength.
  If no ambiguity is possible we denote by $xz$ a minimizing geodesic connecting the points $x,z\in X$.

\section{On  spaces of normal Jacobi fields}
\subsection{Space of normal Jacobi fields} Let  $\gamma:I\to M$ be a geodesic. 
 Denote by   $\mathrm{Jac} (\gamma)$  the $(2n-2)$-dimensional space  of all normal Jacobi fields along $\gamma$.   For any $t\in I$, the map $J\to (J(t),J'(t))$  idenitfies  $Jac (\gamma)$
with the double   $\gamma^{\perp} (t)\oplus \gamma ^{\perp} (t)$ of the normal space $\gamma^{\perp} (t)$ of $\gamma$ at $\gamma (t)$.  This provides  scalar products on $\mathrm{Jac} (\gamma)$, dependent on $t$.  

Moreover, we obtain a  natural smooth manifold structure on 
$$\bigcup _{h\in TM, ||h||\neq 0} \mathrm{Jac} (\gamma ^h) \,,$$
the space of all normal Jacobi fields along all geodesics. We will always refer to this smooth structure, when talking about   continuous or $\mathcal C^k$-dependence of some families of Jacobi fields.

\subsection{$L$-Jacobi fields and general Lagrangians}
Let  $\gamma:I\to M$ be a geodesic. On the space  $\mathrm{Jac} (\gamma)$,
 there is  the  symplectic form  
\begin{equation} \label{eq: w}
\omega (J_1,J_2):=\langle J_1(t) , J'_2 (t) \rangle - \langle J_2(t), J'_1(t) \rangle \;, 
\end{equation}
 independent of the time $t\in I$, \cite{Wilking}, \cite{Jacobi}. 
We will consider \emph{Lagrangians} $\mathcal W$ with respect to this symplectic form $\omega$, thus $(n-1)$-dimensional subspaces of $\mathrm{Jac}$ to which $\omega$  restricts as the $0$-form.

A typical example of such Lagrangians is the space $\mathcal W$ of  \emph{$L$-Jacobi fields}, where  $\gamma$ is a geodesic starting normally on a $\mathcal C^2$-submanifold $L$ of $M$.    This is  the space of variational fields of all variation of $\gamma$ by geodesics starting normally on $L$ and having the velocity of $\gamma$.

In the above notations, the space $\mathcal W$ of $L$-Jacobi fields consists of 
all normal Jacobi fields along $\gamma$ with initial conditions
\begin{equation} \label{eq: initial}
J(0)\in T_pL  \; \; \text{and} \; \; J'(0)  ^{\top} = S^h (J(0))\,.
\end{equation}
 Here
$^{\top}$ is the  orthogonal projection onto   $T_pL$.

\subsection{Focal points} \label{subsec: focalpoints}
Let $\mathcal W$ be a Lagrangian of normal Jacobi fields along $\gamma:I\to M$.
For any $t\in I$, we have the evaluation map $\mathcal{E}^t:\mathcal W\to \gamma^{\perp} (t)$,
which sends the  Jacobi field  $J$ to the vector $\mathcal E^t (J):=  J(t)$.  For 
all but discretely many  $t\in I$, the evaluation map $\mathcal {E}^t$ is an isomorphism.

A  time $t\in I$  is called a \emph{$\mathcal W$-focal time} if $\mathcal E^t$ has a non-trivial kernel. In this case the kernel of $\mathcal E^t$ is denoted by $\mathcal W^t:=\ker (\mathcal E^t)$, its dimension $\dim (\mathcal W^t)$ is called the \emph{ $\mathcal W$-focal multiplicity} of $t$. Elements  of $\mathcal W^t$ are called Jacobi fields (in $\mathcal W$) \emph{focalizing at the times $t$}.

By $^o{\mathcal W}$  we denote the subspace of $\mathcal W$ generated by all focalizing subspaces  $\mathcal W^t,  t\in I$. We call it the \emph{focal-generated subspace} of $\mathcal W$.

If $\mathcal W$ is the Lagrangian of $L$-Jacobi fields along $L$, then $0$ is  $\mathcal W$-focal  with multiplicity $n-1-\dim (L)$.  Other times $0\neq t \in I$ are focal times if and only if $\gamma (t)$ is a focal point of $L$ along $\gamma$ in the sense of Riemannian geometry, thus, if the normal exponential map on the normal bundle $T^{\perp} L$ of $L$ is degenerate at the vector $t\cdot h\in T^{\perp} L$.

The following result was proved in  \cite[Corollary 1.10]{Wilking}
using transversal Jacobi equations, which we will recall later in Section \ref{subsec: trans}.  

\begin{thm} \label{thm: GAG}
Let  $M$ be non-negatively curved,  let $\gamma :\mathbb R\to M$ be a geodesic and let  $\mathcal W$ be a Lagrangian of normal Jacobi fields along $\gamma$. Let $^o{\mathcal W}$ be the focal-generated subspace  of $\mathcal W$ and let $\mathcal W^{par}$ be the subspace all parallel Jacobi fields in $\mathcal W$.

  Then
there is a decomposition $\mathcal W=$  $^o{\mathcal W}  \oplus \mathcal W ^{par}\,$.  Moreover, for  all $t$, all $J_0\in$ $^o{\mathcal W}$  and $J_1 \in \mathcal W ^{par}$, the vectors  $J_0(t)$ and $J_1(t)$ are orthogonal.

\end{thm}

\subsection{Variations  of  focalizing spaces of Jacobi fields} 
Let   $g_i$ be a sequence of smooth metrics, $\mathcal C^2$-converging to $g$. Let
 $\gamma_i:[a,b]\to (M,g_i)$ be a sequence of geodesic converging to $\gamma$. Let 
 $\mathcal W_i$ be a Lagrangians of normal Jacobi fields along $\gamma_i$ converging to a Lagrangian $\mathcal W$ along $\gamma$.   

Denote  by $^o{\mathcal W}$ and $^o{\mathcal W}_i$ the focal-generated subspaces of $\mathcal W$ and $\mathcal W_i$.

\begin{lem} \label{lem: indexsemi}
In the notations above, assume in addition that the boundary points $a$ and $b$ are non-focal points for $\mathcal W$.
If there is just one $\mathcal W$-focal time $t_0\in [a,b]$ then   $^o{\mathcal W} _i$ converge to  $^o{\mathcal W}$.
In general, $^o{\mathcal W}$ is contained in any limit  $\hat {\mathcal W}$ of any  subsequence of    $  ^o {\mathcal W} _i$. 
\end{lem}

\begin{proof}
Assume first  $^o{\mathcal W} =\mathcal W^{t_0}$. 

For any sequence $s_i\in [a,b]$ and any limit $J\neq 0$ of Jacobi fields  $J_i\in \mathcal W_i^{s_i}$, we have $J(t)=0$, where $t=\lim (s_i)$.  Thus, for any $\mathcal W_i$-focal times 
$s_i$, we must have $\lim (s_i)=t_0$. Moreover, 
  the focalizing spaces 
$\mathcal W_i^{s_i}$ subconverge to a subspace of $\mathcal W^{t_0}$.

On the other hand, $\Sigma _{s\in [a,b]}  \dim (W_i^s)$ is equal to $\dim (\mathcal W^{t_0})$, for all $i$ large enough,  since these numbers are the continuously varying \emph{ indices} of the Lagrangians on the interval $[a,b]$, \cite[Proposition 1.4]{Jacobi}.  

For different $\mathcal W_i$-focal times $r_i,s_i \to t_0$, the limits of the corresponding  subspaces $\mathcal W_i^{s_i}$ and $\mathcal W_i ^{r_i}$ are orthogonal subspaces of the space  $\mathcal W^{t_0}$ with respect to a natural scalar product \cite[p.331]{Jacobi}.

Thus,  for all sufficiently large $i$,  the space $^o{\mathcal W}_i$ has dimension 
$\dim (^o{\mathcal W})= \Sigma _{s\in [a,b]}  \dim (W_i^s)$. As seen above, any limit of any sequence $J_i \in$ $ ^o {\mathcal W}_i$ is contained in $^o{\mathcal W}$.  Thus,
 $^o{\mathcal W} _i$ converge to  $^o{\mathcal W}$.

Assume now that there are more than one $\mathcal W$-focal time in $[a,b]$.  Replacing 
$[a,b]$ by  a neighborhood of any $\mathcal W$-focal time  $t_0$ and appplying the previous special case, we deduce that $\mathcal W^{t_0}$ is contained in any limit space $\hat {\mathcal W}$ of a subsequence of $^o{\mathcal W}_i$.  Since this is true for any $t_0$, we deduce that $\hat {\mathcal W}$ contains all of $^o{\mathcal W}$.
\end{proof}

\subsection{Transversal Jacobi equation} \label{subsec: trans}
Let $\gamma:I\to M$ be a geodesic. Let $\mathcal W$ be a Lagrangian of Jacobi fields along $\gamma$ and let $\mathcal V$ be an $l$-dimensional linear subspace of $\mathcal W$. 
In this situation,  \cite{Wilking}  verifies a new Jacobi-equation valid for all 
elements in the quotient space $\mathcal W /\mathcal V$, as we are going to recall.

Outside the closed discrete set of $\mathcal W$-focal times, we have an $l$-dimensional subbundle  $\tilde {\mathcal V }(t):=\mathcal E^t(\mathcal V)$ of the normal bundle $\gamma ^{\perp}(t)=\mathcal E^t (\mathcal W)$ along $\gamma$.   This bundle   $\tilde {\mathcal V}(t)$ uniquely extends to a smooth $l$-dimensional bundle on all of $I$, \cite[p. 1300]{Wilking}.  Denote by $\tilde {\mathcal H} (t)$ 
  the orthogonal complement  of $\tilde  {\mathcal V}(t)$ in $\gamma ^{\perp} (t)$. 

For a section $Y(t)$ in the bundle $\tilde {\mathcal H}$, \cite{Wilking} defines the \emph{transversal covariant derivative} 
\begin{equation} \label{eq: defder}
\tilde {\nabla }  _{\gamma'(t)}  Y (t):=(\nabla  _{\gamma'(t)}  Y (t))^{\perp}
\end{equation} 
 as the projection to $\tilde {\mathcal H}$ of the Levi-Civita  derivative $\nabla  _{\gamma'(t)}  Y (t)$.

Another operator defined  in \cite{Wilking} is the  unique smooth field of linear maps
$\tilde A_t:\tilde V (t) \to \tilde H (t)$ satisfying, for all  $t$ and all  $J\in \mathcal V$:
\begin{equation} \label{eq: defdual}
\tilde A_t(J(t)):=(J'(t))^{\perp}\,.
\end{equation}

Denote by $\mathcal R :\gamma ^{\perp} \to \gamma ^{\perp}$ the Jacobi operator along $\gamma$, and  by $\tilde A_t^{\ast}$ the adjoint operator of $\tilde A_t$. Then the following modified Jacobi-operator  $\tilde {\mathcal R}$ on  sections $Y$ in the bundle $\tilde H$ is defined in \cite{Wilking}:
\begin{equation} \label{eq: defr}
\tilde {\mathcal R}(Y (t)):=(\mathcal R (Y) )^{\perp} + 3 \cdot (\tilde A_t \circ \tilde A_t^{\ast}) (Y (t))\,.
\end{equation}

Due to \cite[Theorem 1.9]{Wilking},  for any Jacobi field $J\in \mathcal W$,
its orthogonal projection $\tilde J $ to $\tilde H$ satisfies along $\gamma$ the equation
\begin{equation} \label{eq: tilde}
\tilde {\nabla } {\tilde  \nabla }  \tilde J + \tilde {\mathcal R} (\tilde J) =0\,.
\end{equation}

Observe now that  $\tilde {\mathcal H}$  is determined by $\mathcal V$ and does not depend on the choice of the Lagrangian $\mathcal W$ containing $\mathcal V$.  Moreover, the formulas \eqref{eq: defder}, \eqref{eq: defdual}, \eqref{eq: defr} depend only on $\mathcal V$.

Thus, given $\mathcal V\subset \mathcal W$ as above, for any Lagrangian $\mathcal W '$ containing $\mathcal V$ and  any  normal Jacobi field $J \in \mathcal W'$ the projection 
$\tilde J$ to $\tilde H$ solves \eqref{eq: tilde}.

  An $l$-dimensional space $\mathcal V$ of normal Jacobi fields  along $\gamma$ is  contained in some  Lagrangian subspace if and only if it is   \emph{isotropic} with respect to the canonical symplectic form $\omega$  on $Jac(\gamma)$.   For such an isotropic $\mathcal V$, the union of all Lagrangians containing $\mathcal V$ is exactly  the $\omega$-perpendicular subspace 
$\mathcal V^{\perp_{\omega}}$ of dimension $(2n-2-l)$.  

As explained above, the projection to $\tilde {\mathcal H}$ sends any $J\in \mathcal V^{\perp_{\omega}}$ to a solution $\tilde J$ of  \eqref{eq: tilde}.  The kernel of this map $J\to \tilde J$ consists exactly of elements of $\mathcal V$.  Since the space of solutions of \eqref{eq: tilde} has dimension $2n-2-2l$ we arrive at the following:

\begin{lem}
Let $\mathcal V$ be an $l$-dimensional isotorpic subspace of $\mathrm{Jac}(\gamma)$.
Then, in the notations above,  the map $J\to \tilde J$  given by the orthogonal projection
to $\tilde {\mathcal H}$ defines a surjective homomorphism 
 $$\mathcal V^{\perp_{\omega}} \to \widetilde {Jac} (\gamma, \mathcal V)\;,$$
where the image $\widetilde {Jac} (\gamma, \mathcal V)$ is the space of solutions of 
\eqref{eq: tilde}.
\end{lem}

 \section{Smoothness of submanifolds}   \label{sec: smooth}
 \subsection{Normal exponential map} \label{Sec: 8}
Assume  $(k,\alpha)=(1,1)$ or $k\geq 2$ and  $0\leq \alpha \leq 1$.
Let $L$ be    $\mathcal C^{k,\alpha}$-submanifold  of a smooth Riemannian manifold $M$.   The normal exponential map $\exp_L$  is the restriction of the smooth exponential map of $M$  from its domain of definition $\mathcal D \subset TM$  to the 
$\mathcal C^{k-1,\alpha}$-submanifold $\mathcal D\cap \nu L$ of $TM$.  Hence, $\exp _L$ is   $\mathcal C^{k-1,\alpha}$.  

The  submanifold $L$ is $\mathcal C^{1,1}$, hence it has positive reach in $M$, \cite{KL-both}. Therefore, $\exp_L$ is locally biLipschitz around the $0$-section in $\nu L$. Hence, 
$\exp_L$ is  a $\mathcal C^{k-1, \alpha}$-diffeomorphism from a neighborhood of the $0$-section in $\nu L$ onto a neighborhood $O$ of $L$ in $M$.

The projection $\Pi ^L:O\to L$ corresponds under the diffeomorphism $\exp _L$ to the projection onto the $0$-section, hence $\Pi^L:O\to L$ is $\mathcal C^{k-1,\alpha}$.

Consider the distance function $d_L$ and its square $d_L ^2$ on the neighborhood $O$ of $L$ defined above.  The gradient $X$ of $d_L^2$ equals $0$ on $L$. For $p\in O\setminus L$ the gradient $X(p)$ equals 
$$X(p)=\Phi _1 (\exp_L^{-1} (p))\,,$$
where $\Phi_t(v)$ denotes the smooth geodesic flow of the manifold $M$.
Therefore, $X$ is $\mathcal C^{k-1,\alpha}$ on $O$.  We deduce, that  $d_L^2$ is $\mathcal C^{k,\alpha}$ on $O$ and $d_L$ is $\mathcal C^{k,\alpha}$ on $O\setminus L$.

\subsection{Improving regularity}  For $k\geq 1, 0\leq \alpha \leq 1$, 
  let $L$ be a $\mathcal C^{k,\alpha}$-submanifold of a smooth Riemannian manifold $M$.  Then the restriction of the tangent bundle $\pi:TM\to L$ of $M$ to $L$ is a
$\mathcal C^{k,\alpha}$-bundle and the notion of  $\mathcal C^{k,\alpha}$-subbundles of $TM$ is well-defined.  A subbundle $W\to L$ is $\mathcal C^{k,\alpha}$ if and only if there exist $\mathcal C^{k,\alpha}$-sections $L\to TM$ with values in $W$, which generate the fiber of $W$ over each point of $L$.

Recall that for any $\mathcal C^{1,1}$ submanifold $L$ of $M$ the second fundamental form $\mathrm {II}_L$ is well-defined at all points of $L$, at which $L$ is twice diffentiable, hence at almost all points of $L$.  Thus, the mean curvature vector field $\mathfrak {h}:L\to T^{\perp} L$ is well-defined almost everywhere.
Moreover, at all such points $x\in  L$ and all normal vectors $0\neq h\in T^{\perp}_xL$, the 
Lagrangian of $L$-Jacobi fields along $\gamma ^h$ is well-defined by   
 \eqref{eq: initial}.

 The following  is a geometric version of the trivial observation that a map is $\mathcal C^{k+1}$ if its derivative is $\mathcal C^k$.

\begin{thm} \label{thm: normalfields}
Let $L$ be a $\mathcal C^{1,1}$ submanifold of $M$.
Let $k\geq 1, 0\leq \alpha \leq 1$.  If $L$ is  a
  $\mathcal C^{k,\alpha}$ submanifold then  the following are equivalent:
\begin{enumerate}
\item The submanifold $L$ is $\mathcal C^{k+1,\alpha}.$
\item The tangent bundle $TL \subset TM$ of $L$ is $C^{k,\alpha}.$
\item The normal bundle $T^{\perp} L \subset TM$ of $L$ is $C^{k,\alpha}.$
\item The second fundamental form $\mathrm{II}_L$ of $L$ in $M$ is $\mathcal C^{k-1,\alpha}$.
\end{enumerate}
If $0< \alpha < 1$ then the above  four conditions are equivalent to the statement that  the mean curvature vector $\mathfrak h:L\to T^{\perp} L$ is $\mathcal C^{k-1,\alpha}$.
\end{thm}

\begin{proof}
Since taking orthogonal complements is a smooth operation,  (2) and (3) are equivalent.  Clearly, (1) implies (2), and (2) and (3) imply (4).  Moreover, if the second fundamental form  $\mathrm {II}_L$ is $\mathcal C^{k-1,\alpha}$ then so is its trace, the mean curvature $\mathfrak h $.

  For the remaining implications, we use the Nash embedding theorem to replace $M$ by a Euclidean space.  The statements are local and we may assume that $L$ is the graph of 
  a $\mathcal C^{k,\alpha}$-map  $F:O\to \R^{n-l}$, for an open set $O\subset \R^l$.

If $\mathrm {II}_L$ is $\mathcal C^{k-1,\alpha}$ then   $\mathrm {II} (\partial _i F, \partial _j F)$  is a $\mathcal C^{k-1,\alpha}$-map from $O$ to $\R^{n-l}$, for all $1\leq j\leq l$.  Since the Jacobian of $F$ is $\mathcal C^{k-1,\alpha}$ as well, the Hessian of $F$ has to be $\mathcal C^{k-1,\alpha}$.  Thus, $F$ is $C^{k+1,\alpha}$. 
Hence (4) implies (1).

Assume now that $0<\alpha <1$ and that the mean curvature $\mathfrak h$ is  $\mathcal C^{k-1,\alpha}$.  The coordinate-wise
Laplace operator satisfies 
\begin{equation} \label{eq: laplace}
\Delta F  (p)=l\cdot \mathfrak h (F(p))\;.
\end{equation}
 By the theorem of Sabitov--Sheffel,  deTurck--Kazhdan \cite{Sabitov}, \cite{KZ},  harmonic coordinates on $L$ are  $\mathcal C^{k+2 , \alpha}$ and the Riemannian metric  tensor $g$ is $\mathcal C^{k,\alpha}$ in these coordinates.
  Applying elliptic regularity, to the equation \eqref{eq: laplace},
we deduce, that $F:L\to \R^n$ is $\mathcal C^{k+1,\alpha}$ in harmonic coordinates.
Hence $L$ is a $\mathcal C^{k+1,\alpha}$-submanifold of $\R ^n$.
\end{proof}

The first part of the  above proof shows the following statement, which is formally not a special case of Theorem \ref{thm: normalfields}.

\begin{cor} \label{cor: normalfields}
Let $L$ be a $\mathcal C^{1,1}$-submanifold of  $M$.
If the almost everywhere defined second fundamental form $\mathrm{II}_L$ extends to a continuous tensor on  $L$ then the submanifold  $L$ is  $\mathcal C^2$.
\end{cor}

\subsection{Smoothness of spaces $L$-Jacobi fields} We restate Theorem \ref{thm: normalfields} and Corollary \ref{cor: normalfields} in terms of Lagrangians of Jacobi fields:

\begin{cor} \label{cor:  contdepend}
Let $L$ be a $\mathcal C^{k,\alpha}$ submanifold of  $M$, where $(k,\alpha)=(1,1)$ or $k\geq 2$ and $0\leq \alpha \leq 1$.
Let  $h^1,...,h^m$  be  $\mathcal C^{k-1,\alpha}$-normal fields on $L$, which span the normal space at all points. For $1\leq i\leq m  $ and $z\in L$, denote   by $\mathcal W^i_z$ the Lagrangian space of $L$-Jacobi fields along the geodesic $\gamma ^{ h^i (z)}$ starting at  $z$ in the direction  $ h^i (z)$.  

\begin{enumerate}

\item If $k\geq 2$ and the families  $z\to \mathcal W^i_z$ are $\mathcal C^{k-1,\alpha}$ then $L$ is a $\mathcal C^{k+1,\alpha}$ submanifold of $M$.

\item If $(k,\alpha)=(1,1)$ and the almost everywhere defined families $z\to \mathcal W^i_z$ extend to families of Lagrangians along $\gamma ^{ h^i (z)}$ continuously defined for all $z\in L$ , then $L$ is a $\mathcal C^2$ submanifold. 

\end{enumerate}

\end{cor}

\begin{proof}
For $1\leq i\leq m$  and  $z\in L$, the space $\mathcal W^i_z$ contains the $(m-1)$-dimensional subspace $\mathcal  U^i_z$ consisting of normal Jacobi fields $J$  with inital conditions $J(0)=0, J'(0)\in T_z^{\perp}L$.  The map 
 $z\to \mathcal U^i_z$ depends in a $\mathcal C^{k-1,\alpha}$ way on $z$ (and so does the family of normal spaces).

Thus, also the orthogonal complement $\mathcal V^i_z$ of $\mathcal U^i_z$ in $\mathcal W^i_z$ (taken with respect to the canonical scalar products at time $0$) depends continuously on  $z$, (respectively, $\mathcal C^{k-1,\alpha}$ for $k\geq 2$).   The evaluation map $J\to J(0)$ defines 
a bijective isomorphism from $\mathcal V^i_z \to T_zL$, which depends continuously (respectively,  $\mathcal C^{k-1,\alpha}$)  on $z$.    

The inverse map assigns to a vector $v\in T_zL$  a normal Jacobi field $J_v$ along the geodesic $\gamma ^{ h^i (z)}$ and this assignement is continuous (respectively, $\mathcal C^{k-1,\alpha}$). 
Thus also the map $$v\to J_v'(0)= S^{ h_i (z)} (v)$$ 
is continuous (respectively, $\mathcal C^{k-1,\alpha}$). 

Thus, for all $i$,  the second fundamental form $\mathrm{II}^{ h_i}_L$ in the direction of $ h_i$ is  $\mathcal C^{k-1,\alpha}$, respectively, has a continuous extension to all of $L$.   
Since $h_i(z)$ constitute a basis of $T_z ^{\perp}L$,  the second fundamental form $\mathrm{II}$ of $L$ is  $\mathcal C^{k-1,\alpha}$ (respectively, continuous). 

Now the result follows from Theorem \ref{thm: normalfields} and Corollary \ref{cor: normalfields}.
\end{proof}


\section{Basics on submetries}

\subsection{(Local) Submetries} \label{subsec: loc}
Recall that
 $P:X\to Y$ is a  \emph{submetry} if for any $x\in X$ and any  $r>0$ the equality
$P(B_r(x))= B_r(P(x))$ holds. 

The map $P$ is called a \emph{local submetry} if for any $x\in X$ there exists some $s>0$ such that the condition $P(B_r(z))= B_r(P(z))$ holds true for any $z\in B_s(x)$ and any $r<s$. 
 We call $X$ the \emph{total space} and $Y$ the \emph{base} of  the local submetry $P$.

A restriction of a (local) submetry $P:X\to Y$  to an open subset $O\subset X$ is a local submetry $P:O\to Y$.  A local submetry $P:X\to Y$ is a  submetry, if   $X$ and $Y$ are length spaces and $X$ is proper \cite[Corollary 2.9]{KL}.

Let $P:X\to Y$ be a local submetry and  let $X$ be a length space.
Replacing  $Y$ by $P(Y)$ 
 we may  assume that the local submetries are surjective.  
Replacing the metric on $Y$ by the induced length  metric, $P$ remains a local submetry
\cite[Corollary 2.10]{KL}.
 Thus, we may  assume without loss of generality that the base space $Y$ is a length space.

For a local submetry $P:X\to Y$, a rectifiable curve $\gamma :I\to X$  is \emph{horizontal} (with respect to $P$) if $\ell (\gamma)=\ell (P\circ \gamma)$.

\subsection{Structure of the base} \label{subsec: basestructure}
Let $M$ be a  smooth Riemannian manifold.
Let $P:M\to Y$ be a surjective local submetry.  Let $y\in Y$ be arbitrary.
 Then 
there exists some $r>0$ such that the closed ball
$\bar B_r(y)$ is an Alexandrov space of curvature $\geq \kappa$, for some $\kappa =\kappa (y)$.  In particular, $\bar B_r (y)$ is convex in $Y$  \cite[Proposition 3.1]{KL}.
Moreover, any geodesic starting in $y$ can be extended  to a minimizing geodesic  connecting $y$ with  the distance sphere  $\partial  B_r(y)$
\cite[Theorem 1.3]{KL}.
 In this case we will say that the \emph{injectivity radius  at $y$ is at least $r$}. 


Set $ m=\dim (Y)$.
Then  $Y$ admits a canonical  decomposition  $Y= \cup _{e=0} ^m Y^e$ into \emph{strata} $Y^e$.  Here, $Y^e$ is the set of all points $y\in Y$, for which  the tangent space $T_yY$ splits off $\R^e$ but not $\R^{e+1}$ as a direct factor.  
$Y^e$ is an $e$-dimensional  manifold with a canonical $\mathcal C^{1,1}$-atlas, which is  locally convex
 in $Y$, \cite[Theorem 1.6]{KL}. 
The distance on $Y^e$ is given  by a Lipschitz continuous Riemannian metric;  the tangent space $T_yY^e$ is the maximal Euclidean factor of 
$T_yY$ \cite[Theorem 11.1]{KL}.   Due to 
 \cite[Corollary 1.2]{L},   any stratum $Y^e$ has curvature locally bounded from above. (This is not used below and  will follow from Theorem \ref{thm: new3}).

The maximal stratum $Y^m$ is open and dense in $Y$. Points in $Y^m$ are  called \emph{regular} points of $Y$. Fibers $P^{-1} (y)$ with  $y\in Y^m$ are called \emph{regular fibers} and  points $x\in P^{-1} (Y^m)$ are called \emph{regular points} of $M$.

Due to \cite[Lemma 10.1, Theorem 11.1]{KL}, for any point $y\in Y^e$, there exists some $r_0=r_0 (y)>0$ such that 
\begin{equation} \label{eq: injradius}
  \text{The injectivity radius  at any }  y' \in B_{r_0} (y)\cap Y^e \text{ is at least } r_0.
\end{equation}

\subsection{Structure of the fibers} \label{subsec: strfib}
A locally closed subset $L\subset M$ has
\emph{positive reach} in $M$ if the closest-point projection $\Pi^L$ is uniquely defined on a neighborhood $U$ of $L$ in $M$.   In this case,  $\Pi^L$ is locally Lipschitz on $U$ and the distance function $d_L$ is $\mathcal C^{1,1}$ on $U\setminus L$ \cite{KL-both}.


  Any set $L$  of positive reach is a topological manifold if and only if $L$ is a $\mathcal C^{1,1}$ submanifold of $M$, \cite[Proposition 1.4]{Ly-conv}.  On the other hand, any set $L$ of positive reach contains a subset $L'$ open and dense  in $L$, which is a $\mathcal C^{1,1}$-submanifold, possibly with components of different dimensions \cite[Theorem 7.5]{Rataj}.

Let $P:M\to Y$ be a surjective local submetry.  Let $y\in Y^e\subset Y$.  Then the  fiber $L=P^{-1} (y)$ and the preimage
$S=S^e=P^{-1} (Y^e)$  are subsets of positive reach in $M$ \cite[Theorems 1.1,  1.7]{KL}.

Neither $L$ nor $S$ have to be manifolds. 
However, for every $y\in Y\setminus \partial  Y$ the fiber $L=P^{-1} (y)$ is   a $\mathcal C^{1,1}$-submanifold of
$M$ \cite[Theorem 1.8]{KL}.  In particular, this applies to all  regular fibers.

For $y\in Y^e$ and $r_0>0$ as in \eqref{eq: injradius}, let
$x\in L=P^{-1} (y)$ and $4r<r_0$ be such that $\bar B_{4r} (x) $ is compact.
Then, for any $y'\in B_{r}(y) \cap Y^e$, the closest-point projection $\Pi^{L'}$ to $L'=P^{-1} (y')$   is uniquely determined on $B_r(x)$.  Moreover, on $B_r(x)$, we have (\cite[Corollary 2.9]{KL})
$$d_{L'} =d_{y'} \circ P\;.$$  
 Hence, on $B_r(y) \setminus L'$ the function  $d_{L'}$  is $\mathcal C^{1,1}$ and its gradient $\nabla d_{L'}$ is sent by the differential $DP$  to the gradient of $d_{y'}$ \cite[Section 2.5]{KL}.

\subsection{Infinitesimal structure}
Let  $P:M\to Y$ be  a local submetry, let $x\in M$  be arbitrary,  $y=P(x)$ and denote by $L$ the fiber
$P^{-1} (y)$.

There exists a  differential $D_xP:T_xM\to T_yY$, which is itself a submetry.
The tangent  space $T_xL$ is  the preimage $(D_x P)^{-1} (0)$ and it is a convex cone in  $T_xM$  \cite[Proposition 3.3, Corollary 3.4]{KL}.
 We call   $T_xL$  the \emph{vertical space} at  $x$ and denote it by $V^x$.

The \emph{horizontal space} $H^x$ is the dual cone of $T_xL$ in $T_xM$.  The cone  $H^x$ consists of all  $h\in T_xM$ such that
$||h||= |D_xP(h)|$, where  $|\cdot |$ on the right side denotes the distance  to the origin of $T_yY$.

A Lipschitz curve $\gamma:I\to M$ is horizontal if and only if the vector $\gamma'(t)$ is horizontal, for almost all $t\in I$.

  \subsection{Transnormal submetries}
    A local submetry $P:M\to Y$ is  \emph{transnormal} if  all fibers  of $P$ are topological manifolds (with empty boundary!).  This happens if and only if all
fibers are $\mathcal C^{1,1}$-submanifolds.   $P$ is transnormal if and only if all vertical spaces $V^x$ are Euclidean spaces.  Another equivalent condition is that
a  geodesic $\gamma$ in $M$ with horizontal $\gamma'(0)$ 
 stays horizontal for all times, \cite[Section 12]{KL}.

Let $P:M\to Y$ be a local transnormal submetry.
Let $y\in Y$ be a point in the stratum $Y^e$. Let $L=P^{-1} (y)$ and $S=P^{-1} (Y^e)$.  Let $\gamma :[0,a]\to M$ be a horizontal geodesic starting at $x=\gamma (0)\in L$ with $\gamma '(0)\in T_xS$.  
Then, for all but finitely  many times $t\leq a$, the point $\gamma (t)$ is contained in the stratum $S$, \cite[Corollaries 1.4, 6.3]{L}.
In particular, if $\gamma$ contains a regular point, then all but discretely many points 
on $\gamma$ are regular.     We such $\gamma$ a \emph{regular horizontal geodesic}.

Let $\gamma_{1,2} :I\to M$ be  horizontal geodesics with projections $\bar \gamma _i:=P\circ \gamma _i$.  If at some $t_0\in I$ we have 
$\bar \gamma _1(t_0)= \bar \gamma _2 (t_0)$ and $\bar \gamma _1 ' (t_0)=\bar \gamma _2 ' (t_0)$, then $\bar \gamma _1$ and $\bar \gamma _2$ coincide \cite[Proposition 12.7]{KL}, \cite[Lemma 7.1]{L}.

\section{Basic normal fields} \label{sec: basicnormal}

\subsection{Finer infinitesimal structure} \label{subsec: finer}
Let  $y=P(x)$ be contained in  the stratum $Y^e$.  For any vector $\bar h\in \mathbb R^e=T_yY^e$, the horizontal lift  $h$ of $\bar h$  in $T_xM$ is unique.   The set of all such lifts $h\in T_x M$ is an $e$-dimensional Euclidean subspace $H^x_0$ and the differential $D_xP$   respects  the decomposition
$$D_xP:T_xM =H^x_0 \times (H^x_0)^{\perp} \to T_yY = T_yY^e \times (T_yY^e)^{\perp}\;,$$
\cite[Proposition 5.6]{KL}.  
    The space $T_yY^e$ resp. $H^x_0$ is spanned by the gradients of the distance functions $d_{y'}$ resp. of $d_{L'}=d_{y'}\circ P$ with $y'$ in a neighborhood of $y$ in $Y^e$.  
 Thus,  the $e$-dimensional distribution $z\to H^z_0$ on the stratum $S=P^{-1} (Y^e)$ is Lipschitz continuous.

Let $x\in L$ be chosen such that  a neighborhood of $x$ in $L$ is a $\mathcal C^{1,1}$ submanifold of $M$. Then a neighborhood $U$ of $x$ in $S=P^{-1} (Y^e)$ 
is a $\mathcal C^{1,1}$ submanifold of $M$ as well.  In this case, the restriction $P:U\to Y^e$ is a $\mathcal C^{1,1}$ Riemannian submersion, \cite[Proposition 1.7]{KL}.  In particular, the $\dim (L)$-dimensional distribution $z\to V^z:=T_z L^z$
is Lipschitz continuous (here and below we denote by $L^z$ the fiber of $P$ through $z$).  The space $ H^z_0$ is the normal space of $V^z$ within $T_zS$.

The normal bundle $H^z_1:= T _z^{\perp} S$ of $S$ is the orthogonal complement of $H^z_0$ in $H^z$ and is  Lipschitz continuous on $U$ as well.

  For any  $z\in U$ and 
$w=h_0+h_1+v \in T_zM= H_0^z\oplus H_1^z\oplus V^z$, we have 
$$D_zP(w)= D_zP(h_0+h_1)=(D_zP(h_0), D_zP(h_1))\in T_yY^e \times  (T_yY^e)^{\perp}= T_yY\,.$$

We mention two special cases. If $e=m=\dim (M)$ then $T_yY^e=T_yY$, the spaces $H^z_1$ and $(T_yY^e)^{\perp}$ are singletons.

If $e=\dim (M)-1$, then $(T_yY^e)^{\perp}$ is  the ray $[0,\infty)$ and the restriction 
$D_zP :H_1^z\to T_yY^e=[0,\infty)$ sends any vector $h$ to $ ||h|| $.
Thus, in this case, the submetry $D_zP:H^z\to T_yY$ has  as fibers concentric spheres 
in affine spaces parallel to $H_1^z$.

\subsection{Basic normal fields} \label{subsec:setting}
Let $y\in Y^e \subset Y$,  $L =P^{-1} (y)$ and  $S:=P^{-1} (Y^e)$ be as above.  
For any $\bar h\in \mathbb R^e =T_yY^e$ and any $z\in L$, there exists a unique 
horizontal vector $\hat h _z\in H^z$ with $D_z P(\hat h_z)=\bar h$.  The vector field
$z\to \hat h_z$ along $L$ is called a \emph{basic normal field along} $L$.  

Any basic normal field $\hat w$ is Lipschitz.  The space of all basic normal fields along $L$ is $e$-dimensional, we will denote it by $\mathcal B(L)$.  For any $z\in L$, the map $\hat h \to \hat h_z$ is an isomorphism between $\mathcal B(L)$ and $H^z _0\subset H^z$.   

For any $\hat h\in \mathcal B(L)$ the norm $ ||\hat h_z||$ coincides with the length
of the image vector $\bar h  =D_zP(\hat h_z)\in T_yY^e$ and it is independent of $z$.

\newpage
\centerline {\bf II.  Smoothness of the quotient space}

\section{Good points,  holonomy maps and holonomy fields} \label{sec: good}

\subsection{Holonomy maps} \label{subsec: holmap2}
Let $P:M\to Y$ be a local submetry, $L$ be a fiber  $P^{-1} (y)$. Let a basic normal field $\hat h\in \mathcal B(L)$ be fixed.   For any $z\in L$, consider the geodesic  $\gamma ^{\hat h_z}$ starting in the direction $\hat h_z$.
 For $r\in \mathbb R$, denote by $\mathcal D^r=\mathcal D^r(\hat h)$ the set of all  $z\in L$ such that $\gamma ^{\hat h _z}(r)=\exp (r\cdot \hat h_z)$ is defined. $\mathcal D^r$ is open in $L$; if $M$ is complete then   $\mathcal D^r =L$. 

Let  $r_0=r_0(y)>0$ be as in \eqref{eq: injradius}. Whenever $r_0> ||r\cdot h|| $, the geodesics $\gamma^{\hat h_z}, z\in \mathcal D^r$ are horizontal on $[0,r]$ and are sent by $P$ to the same geodesic in $Y$, \cite[Proposition 7.3]{KL}.

If $P$ is transnormal then the geodesics $\gamma^{\hat h_z}, z\in \mathcal D^r$ are horizontal on $[0,r]$,  for any $r$, and they are sent by $P$ to the same piecewise geodesic in $Y$, \cite[Proposition 12.7]{KL}, \cite[Lemma 7.1]{L}.

In both cases we call  the map $z\to  \gamma^{\hat h_z} (r)$ the holonomy map along $\gamma$ and denote it as 
$$Hol^{\gamma}:\mathcal D^r\to L^{\gamma (r)}\;.$$

If $r_0> ||r\cdot h|| $, the holonomy map $Hol^{\gamma} :\mathcal D^r\to L^{\gamma(r)}$ coincides with the closest-point projection to $L^{\gamma (r)}$. It is
injective and locally biLipschitz;  the inverse map is the closest-point projection to $L$.

If $P$ is transnormal and $r$ arbitrary, then the holonomy map $Hol^{\gamma}:\mathcal D^r\to L^{\gamma (r)}$ is Lipschitz open, \cite[Proposition 7.2]{L}.  If, in addition,  $(P\circ \gamma) (r)$ is contained in the same stratum as $y$ (which happens for all but discretely many times $r$), the holonomy map is injective and locally biLipschitz.

\subsection{Good points}  \label{subsec:good}
Let $P:M\to Y$ be a local submetry, let $L$ be a fiber of $P$.
 We call  $z\in L$ a \emph{good point of the fiber $L$} if 

\begin{enumerate}
\item
A neighborhood of $z$ in $L$ is a $\mathcal C^{1,1}$-submanifold of $L$. 
\item    $L$  is twice-differentiable at $z$, thus $L$ coincides at  $z$  with a smooth submanifold up to second order.
\item For a fixed basis $\hat h^1,... \hat h^l$ of basic normal fields along of $L$, the Lipschitz vector fields $\hat h^j$ are differentiable at $z$ for $1\leq j\leq l$.
\end{enumerate}

  Any basic normal field $\hat w$ is differentiable at any good point  $z$. Thus, $(v, \hat h)\to \nabla _v \hat h$ is  a bilinear map $V^z\times \mathcal B(L) \to T_zM$.

By the  Rademacher theorem, almost every point in every manifold fiber of $P$ is a good point of its fiber.
Since any fiber $L$  contains open dense $\mathcal C^{1,1}$-submanifolds,  the set of good points is dense in $L$.

At any good point $z\in L$,   the second fundamental form $\mathrm{II}_{z,L}$  is well-defined.  For any $h\in H^z=T^{\perp}_zL$, the shape-operator $S^h_L$ is well-defined by
 \eqref{eq: defsh}, and so is the Lagrangian of $L$-Jacobi fields along the geodesic $\gamma ^h$.  We will denote this  Lagrangian by $\mathcal W(h)$.

\subsection{Holonomy fields}  \label{subsec: goodpoints}
Let $z\in L$ be a good point of the fiber $L=P^{-1} (y)$.  Let $0\neq \hat h$ be a basic normal field along $L$, $h=\hat h_z\in H^z$ and let $\gamma=\gamma^h$ be the geodesic starting in the direction of $\gamma$.    

For $v\in V^z$ we consider the  Jacobi field along $\gamma$ defined by
\begin{equation} \label{eq: jacinit}
 J_{h, v} (0)=v \;; \;  J'_{h , v}(0)=\nabla _v \hat h \;.
\end{equation}
We call such $J_{h,v}$ a \emph{holonomy field} and denote the $\dim (L)$-dimensional space of all such Jacobi fields  along $\gamma ^h$ by $\mathcal I^h$.

The Jacobi field $J_{h,v}$ is the variational field of the variation 
\begin{equation} \label{eq: correspondence} 
\gamma _s(t):=\exp _{\eta (s)} (t\cdot \hat h _{\eta (s)})\;,
\end{equation}
where $\eta $ is an arbitrary Lipschitz curve in $L$ with $\eta(0)=z$ and $\eta'(0)=v$.
From \eqref{eq: correspondence} we directly see, that $J_{h,v}$ is an $L$-Jacobi field
along $\gamma^h$.  Moreover, for all $r$  smaller than the injectivity radius at $y$,
the value  $J_{h,v} (r)$ is tangent to the fiber of $P$ through $\gamma (r)$, hence
\begin{equation} \label{eq: verti}
J_{h,v}(r)\in V^{\gamma (r)}\;.
\end{equation}
If $P$ is transnormal, \eqref{eq: verti} holds true holds true for all $r$. 

By definition, $J_{h,v} (r)$ coinncides with the differential in the direction  $v$ of the 
holonomy map $Hol^{\gamma} : \mathcal D^r\to L^{\gamma(r)}$
\begin{equation}
 D_z Hol^{\gamma} (v)=J_{h,v} (r)\;.
\end{equation}
Since the holonomy map is Lipschitz-open, the differential must be surjective.   Thus, we have verified:

\begin{lem} \label{lem: consequent}
Let $z$ be a good point of a fiber $L=P^{-1} (y)$ of a local submetry $P$. Let $\gamma=\gamma ^h:I\to M$ be a horizontal geodesic starting at $z$, such that $h$ is tangent to the stratum $S$ of $z$.  Let   $r\in I$ be given.  Assume that $r\cdot ||h||$    is  smaller than the injectivity radius of $P(z)$  or that $P$ is transnormal.  Then  $v\to J_{h,v} (r)$ is a surjective linear map $V^z\to V^{\gamma (r)}$, which is the differential of the holonomy map along $\gamma$.
\end{lem}

\section{Smoothness of  strata}

We generalize Theorem \ref{thm: new3} and  the first part of  Theorem \ref{thm-sm1}:

\begin{thm} \label{thm-sm3}
Let $M$ be a smooth Riemannian manifold, let $P:M\to Y$ be a surjective local submetry.
Then any stratum $Y^e$ is  isometiric to a smooth Riemannian manifold.
\end{thm}

\begin{proof}
Let $y\in Y^e$ be arbitarary. 
It is sufficient to find  some smooth coordinates around $y$ in $Y^e$, such that the distance around $y$ is defined  by a smooth Riemannian metric in these coordinates.  Then distance functions to points are smooth within the injectivity radius. Hence, distance coordinates define a smooth atlas on $Y^e$, with respect to which the Riemannian metric is smooth.  

Set $L=P^{-1} (y)$ and $S=P^{-1} (Y^e)$. 
We fix a \emph{good point} $z\in L$  and choose $r>0$, smaller than the injectivity radius of $z$ in $M$ and  than the injectivity radius of $y$ in $Y$.
Let $O_z$ be the set of  vectors of norm less than $r$ in the normal space $H^z_0$ to $L$ in $S$. We  set $N=\exp _z (O_z)$.  

Making $r>0$ smaller if needed, we may assume that $B_r(z)\cap S$ is a $\mathcal C^{1,1}$-manifold.  Then $N$ is a smooth submanifold of $M$ contained in $S$.  Decreasing $r>0$ further, if needed, we may assume that the tangent spaces to $N$ do not contain vertical vectors.
Then the restriction of $P$ to $N$ defines a $\mathcal C^{1,1}$-diffeomorphism between $N$ and the ball $B_{r} (y)\cap Y^e$.

Consider the  pull-back metric $g^{\ast}:=P^{\ast} (g_{Y^e})$ to $N$  of the Riemannian metric
$g_{Y^e}$ on $Y^e$.  It suffices to prove that this, a priori only Lipschitz continuous metric  $g^{\ast}$, is indeed a smooth Riemannian metric on $N$.

 The norm $||u||_{g^{\ast}}$ of a tangent vector $u \in T_qN$ is defined as 
the norm (in the original smooth Riemannian metric on $M$) of the orthogonal projection  of $u$ onto $ H^q$.   This projection is computed by subtracting from $u$ the orthogonal projection of $u$ onto $V^q$.
Thus, in order to deduce the smoothness of the metric $g^{\ast}$, it suffices to verify that the distribution $q\to V^q$ is a  smooth distribution along the submanifold $N$.

However, for $q\in N$ and $h:=\exp^{-1}_z(q)\in H^z$,   we have 
$$V^q= \{J_{h,v } (1)  \;\; : \;\;  v\in V^z  \}\,,$$ 
by Lemma \ref{lem: consequent}.
  Since $h$ depends smoothly on $q$ and Jacobi fields depend smoothly on the initial values, we deduce that $q\to V^q$ is smooth, finishing the proof of the theorem.
\end{proof}

\section{Smoothness in codimension one}

We are going to prove the following strengthening   of Theorem \ref{thm: codim1}:

\begin{thm} \label{thm-sm4}
Let $M\to Y$ be a surjective local submetry with $\dim (Y)=m$. Let $y$ be a point in $Y^{m-1}$.  Then a neighborhood of $y$ in $Y$ is isometric to $M'/\langle g \rangle $, where $M'$ is a smooth $m$-dimensional Riemannian manifold and $\langle g \rangle $ is a two-element group of isometries of $M'$ generated by a reflection  $g$ at a totally geodesic hypersurface in  $M'$. 
\end{thm}

\begin{proof}
First we are going to  find a \emph{slice}: a smooth manifold $N$ 
with a local submetry $P':N\to Y$ onto an open neighborhood of $y$ in $Y$,
such that the $P'$-fiber  of $y$ in $N$ consists  of one point. 

 We set  $L=P^{-1} (y)$ and $S=P^{-1} (Y^{m-1})$ and fix a good point $z\in L$.
Replacing $M$ by a small ball around $z$ and $Y$ by the image of this ball, we may assume  $Y=Y^m\cup Y^{m-1}$.
Moreover, we may assume that all fibers contained in $S$  are $\mathcal C^{1,1}$-submanifolds of $M$.  Then  $P$ is transnormal, since all remaining fibers are regular.

We choose $4\cdot r>0$ smaller than the injectivity radius of $M$ at $z$ and of $y$ in $Y$.
Denote by $\tilde O_z$ the open ball of radius $r$ in the horizontal space $H^z$. Thus, the restriction $\exp _z:\tilde O_z\to M$ is a smooth embedding. Denote by $N\subset M$ ist image.      By construction, the image $P(N)$ is the open ball $B_{r} (y)\subset Y$. Moreover, the preimage of $y$ in $N$ is just the point $z$.  In order to finish the first step, we need  to equip $N$ with a smooth metric, so that the restriction $P:N\to Y$  is a local submetry.

 For every $x\in L$, every $h\in H^x_1=T_x^{\perp}S$ is sent by the differential $D_xP$ to the unique vector of length $||h||$ in $T_yY= \R^{m-1} \times [0,\infty)$ which is orthogonal to $\R^{m-1}=T_y Y^{m-1}$.

 The distributions $x\to V^x$ and $x\to H^x_0$  are Lipschitz continuous along $L$  and differentiable   at $z$, hence so is the complementary distribution $x\to H^x_1$.  Therefore,  we find a neighborhood $U$ of $z$ in $L$ and Lipschitz continuous, pointwise orthonormal  vector fields 
$\hat h_1,...,\hat h_e$ along $U$, spanning $T^{\perp} S$ at all $x\in U$ and being  differentiable at the point $z$.
Thus,  linear combinations of $\hat h_i$,   extend every vector $h\in H^z_1$
to a Lipschitz continuous vector field $\hat h$ along $L$ with values in the distribution $H^x_1= T_x ^{\perp} S$.  Every $\hat h$ is differentiable at $z$, the map $h\to \hat h$ is a linear injection from $H^z_1$ into the space of sections of $T^{\perp} S$ along $U$. Finally, $D_xP (\hat h_x)\in T_y^{\perp} Y^{m-1}$ only depends on $||h||$.

 Hence,  for any $\hat h:U\to T^{\perp} S$ as above and for any basic normal field $\hat w\in \mathcal B(L)$,  the normal vector field $\hat h +\hat w$ along $U$ is \emph{basic} in the sense that $D_x P (\hat h _x+\hat w _x) \in T_yY$ is independent of  $x\in U$, Section \ref{subsec: finer}.

Thus, we have extended any $u\in \tilde O^z$ to a Lipschitz continuous horizontal vector field $\hat u$ along $U$, which is differentiable at $z$ and  such that the value $D_xP (\hat u_x)$ is independent of $x$.  Choosing $r$ smaller, if needed, we may assume that the horizontal geodeisc   $\gamma ^{\hat u_x} :[0,1] \to M$ 
starting in the direction of $\hat u_x$  is defined for all $x\in U$.   

By construction, $P\circ \gamma ^{\hat u_x}$ is a geodesic starting in $y$.  It  is independent of  $x$ and  only depends  on $u$. More precisely, it only depends on 
the projection of $u$ to $H^x_0$ and the  norm of the projection of $u$ to $H^x_1$. 

For fixed $u\in \tilde O_z$, the variational fields of the family $\gamma ^{\hat u_x}, x\in U$, define a $\dim (L)$-dimensional familiy of normal  fields along the geodesic $\gamma ^{u}$. Similarly to  Section \ref{subsec: goodpoints}, these  Jacobi fields along $\gamma ^u$  are  parametrized  as $J_{u,v}$ with $v\in V^z$, where the initial values of $J_{u,v}$ are
$$J_{ u , v }(0)=v\;; \;\;\;   J'_{ u , v}(0)=\nabla _v \hat u \;.$$

As in the previous section, the space 
$$E^{u}:=\{J_{u,v} (1)   \; ;\;  v\in V^z\} \subset V^{\exp _z(u)} \subset T_u M\,,$$
is $\dim (L)$-dimensional, if $r$ is small enough.  Moreover,   $E^u$
depends smoothly on $u\in \mathcal O_z$, hence, on the point $\exp (u)$ in $N$.

By the continuity of the family $E^u$, for any $u\in \tilde O_z$, the space $E^u$
and the tangent space $T_{\exp_z(u)}N$ are complementary to each other. 

For $q\in N$ we denote as $C^q\subset T_qM$ the orthogonal complement of $E^u$, where $u=\exp_z^{-1} (q)$.     By construction, the distribution $q\to C^q$ is smooth along $N$, moreover, the orthogonal projection $n\to \tilde n$ from $T_qN$ to $C^q$ is bijective.  

Since $E^u$ is contained in the vertical space $V^q$, the differential  $D_qP$
restricts to a submetry $D_qP:C^q\to T_{P(q)}Y$. Moreover, by the same reason, $D_qP$  
 commutes with the projection to $C^q$, thus 
$$D_qP (n)=D_qP(\tilde n)\;,$$
for all $q\in N$ and all $n\in T_qN$.

We now define a new smooth Riemannian metric 
 $\tilde g$  on $N$  by setting
$$||n||_{\tilde g}: =|| \tilde n||_g \;,$$
for all $n\in T_qN$. Here, $\tilde n$ is the orthogonal projection of $n$ onto the orthogonal complement of $E^u$.

As observed above, the differential $D_qP :T_qN\to T_{P(q)} Y$  becomes a submetry, when $T_qN$ is considered with the metric $\tilde g$ (hence identified with $C^q$).  Applying \cite[Corollary 1.4]{Lyt-open} we deduce that $P:(N,\tilde g)\to Y$ is a local submetry. This finishes the first step.

Thus, replacing $(M,g)$ by $(N, \tilde g)$, we may assume that $L$ consists just of one point. Making $M$ smaller, if needed, we then deduce that  $P^{-1} (y')$ is a singleton, for every $y'\in Y^{m-1}$.  Thus, $S$  coincides  locally around any of its points  $x$ with the exponential  image of a small ball $H^z_0$.  Thus,
$S$ is locally convex in $M$.  In particular, $S$ is a smooth manifold which is sent by $P$ locally isometrically onto  $Y^{m-1}$.

The normal bundle $T^{\perp} S$ of the smooth submanifold $S$ is smooth.    With $r>0$ as above, set $U=B_r(z)\cap S$ and let $O$ be the open set of vectors $h\in  T^{\perp} U$ with $||h|| <r$.  By the assumption on $r$, the normal exponential map
$\exp_S$ is a diffeomorphism from $O$ onto an open subset $O'$ of $M$.  Moreover, the composition $P\circ \exp _S:O\to Y$ has as fibers concentric spheres in individual normal spaces $T^{\perp}_x S=H^x_1$.   Thus, the decomposition of $O$ into fibers of $P\circ \exp _S$ is a singular foliation, meaning that the decomposition coinicdes with the orbit of the action of some group of smooth diffeomorphisms, see Section \ref{subsec: foliation} below.

Therefore, the restriction of $P$ to $O'$ is a smooth singular foliation. Hence, this restriction is a smooth singular Riemannian foliation, see Section  \ref{subsec: foliation}.   Now we can apply \cite[Proposition 3.1]{Thorb} and deduce that a neighborhood of $y$ in $Y$ is isometric to a smooth Riemannian orbifold.  Since $T_yY$ is a Euclidean halfspace, a small ball around $y$ in $Y$ is isometric to the quotient of a smooth Riemannian manifold by a single isometric  reflection at a totally geodesic hypersurface.
\end{proof}

\section{Some genericity statements}
Let $P:M\to Y$ be a surjective local submetry with $\dim (Y)=m$.
By Theorem \ref{thm-sm4},  $Y':=Y^m \cup Y^{m-1}$  is a smooth Riemannian orbifold. 

 For almost any $y\in Y^m$ and almost every unit direction $w\in T_yY$ there exists exactly one maximal quasi-geodesic starting in the direction of $w$ and this quasi-geodesic is completely contained in $Y'$, \cite[Remark 12.2]{KL}, see also \cite[Theorem 1.6]{KLP} and \cite{Mondino}. 

   Quasi-geodesics in $Y'=Y^m\cup Y^{m-1}$ are uniquely defined by their starting vectors  and they are exactly 
 the orbifold-geodesics of the orbifold $Y'$ parametrized by arclength,  cf. \cite{Lange}.   More precisely, any quasi-geodesic in $Y'$ starting  in a direction contained  in $T_yY^{m-1}$ stays in $Y^{m-1}$ and is a geodesic in the Riemannian manifold $Y^{m-1}$.  All other quasi-geodesics  intersect $Y^{m-1}$ in a discrete set of times, are  geodesics outside of $Y^{m-1}$ and are reflected at $Y^{m-1}$.  
  Any quasi-geodesic $\gamma$ in $Y'$ is an orbifold-geodesic (hence it locally lifts into Riemannian manifolds) and the notions of Jacobi fields, the Jacobi equation and conjugate points along $\gamma$ are well-defined.


 \begin{lem} \label{lem: no-conj}
Let $P:M\to Y$ be a surjective local submetry, $\dim (Y)=m$.
 For almost all $y\in Y^m$ and almost every unit direction $w\in T_yY$ there exists exactly one maximal   quasi-geodesic starting in the direction of $w$.  This quasi-geodesic is contained  in  $Y^m\cup Y^{m-1}$ and it is an orbifold-geodesic in the smooth Riemannian orbifold $Y^m\cup Y^{m-1}$.  Moreover,  no conjugate point to $y$ along $\gamma$ is contained in $Y^{m-1}$. 
\end{lem}

\begin{proof}
 All properties   but the last statement about conjugate points have been explaned above.

Denote by $K$ the set of all unit tangent vectors $w$ in  $TY^m$ such that  the  maximal  quasi-geodesic $\gamma ^w$ starting in the direction of $w$ is contained in $Y^m\cup Y^{m-1}$.    As explained above, $K$ has full measure in the unit tangent bundle of $Y^m$.  Moreover, $K$ is invariant under the measure preserving local  quasi-geodesic flow (which coincides with the geodesic flow of the smooth Riemannian orbifold $ Y^m \cup Y^{m-1}$).

For any $w\in K$, the quasi-geodesic $\gamma ^w$  intersects  $Y^{m-1}$  in at most countably many points.  Each of these countably many points has at most countably many conjugate points along $\gamma^w$.  

Let  $K_0$ be the set of all $w\in K$ such that no point in $Y^{m-1}$ is conjugate to $\gamma^w(0)$ along $\gamma ^w$.  Then $K\setminus K_0$ contains only countably many vectors $(\gamma^w )'(t)$ on the flow-line of $w$ under the geodesic flow.  Trivializing the flow in a neighborhood of any $w\in K$, we see that $K_0$ has full measure in $K$.  This finishes the proof.
\end{proof}

We call a regular point $y\in Y$ a \emph{typical point} and a unit direction $w \in T_yY$ a \emph{typical direction} if they satisfy the statements of Lemma \ref{lem: no-conj}. In this case, we call any positive multiple of $w$ a \emph{typical direction} in $T_yY$.
Hence, almost every point in $Y$ is  typical. For any typical point $y\in Y$,
almost every direction $w$
in  $T_yY$  is  a typical direction.

We will call a point $x\in M$ \emph{typical} if  $x$ is a regular point, $P$ is twice differentiable at $x$ and if  $y=P(x)$ is a typical point of $Y$.
  The set of all typical points will be denoted by $\tilde M\subset M$.
Since the set of regular points has full measure in $M$, the Rademacher theorem for the $\mathcal C^{1,1}$-map $P:M_{reg}\to Y_{reg}$ together with  Lemma \ref{lem: no-conj} and the Fubini theorem, imply that $\tilde M$ has full measure in $M$.

We call a fiber  $L$ of $P$   a \emph{typical fiber} of $P$ if $L\cap \tilde M$ has full measure in the  $\mathcal C^{1,1}$-manifold $L$.  By the Fubini theorem, for almost every $y\in Y$, the fiber $P^{-1} (y)$ is a typical fiber.

If $L=P^{-1} (y)$ is a typical fiber and $0\neq \hat h \in \mathcal B(L)$ is a basic normal field, we call $\hat h$
a \emph{typical basic normal field}, if  for some (and hence any) $x\in L$,  the vector $D_xP (\hat h_x)\in T_yY^m$ is  a typical direction.

We call a horizontal geodesic $\gamma$ in $M$ a \emph{typical horizontal geodesic}, if $\gamma$ contains a typical point $x= \gamma (t)$ such that $(P\circ \gamma)'(t)$ is a typical direction in $Y$.

\section{O'Neill formula and related statements} \label{sec: on}



\subsection{Spaces of Jacobi fields along typical geodesics}  \label{subsec: jacident}
Let $P:M\to Y$ be a transnormal local submetry.  Let $L\subset M$ be a typical fiber  and $x\in L$ a typical point.  Let $0\neq \hat h \in \mathcal B(L)$ be a typical basic normal field and let $h=\hat h_x\in H^x$.  Let $\gamma=\gamma^h$ be the typical horizontal geodesic
with $\gamma'(0)=h$. 
Thus,  $x$ is a regular point, $P$ is twice-differentiable at $x$, $P(x)$ is a typical point in $Y$ and $P\circ \gamma$ is contained in $Y^m \cup Y^{m-1}$.


 Denote by $\tilde {\mathcal U}$ the space of Jacobi fields along $ \gamma$, which arise as variational fields of variations of $\gamma$ through horizontal geodesics.    

$P$ is twice differentiable at $x$, thus the Lipschitz manifold $\mathcal H$ of horizontal vectors at regular points has a well-defined linear tangent space $T_h\mathcal H$ of dimension $n+m$.
Since horizontal geodesics close to $\gamma$ are uniquely  described by their starting vectors,  variational fields $J\in \tilde {\mathcal U}$ are in one-to-one correspondence  $T_{h} \mathcal H$.    Hence $\tilde {\mathcal U}$ is an $(n+m)$-dimensional vector space of Jacobi fields.   

The space $\tilde {\mathcal U}$ contains the two-dimensional space of Jacobi fields  everywhere tangent to $\gamma$.   Hence the pointswise orthogonal complement of this two-dimensional space, the space of \emph{normal} Jacobi fields along $\gamma$ defined by variations through horizontal geodesics, is a vector space of dimension $n+m-2$. We denote this space as $\mathcal U^h=\mathcal U^{\gamma}=\mathcal U\subset \tilde {\mathcal U}$.

As explained in Section \ref{sec: basicnormal}, the space $\mathcal I=\mathcal I^{\gamma}=\mathcal I^h$ of holonomy fields along $\gamma$ is an $(n-m)$-dimensional vector space,   which is  contained in $\mathcal U$. 
 For any $t$ in the domain of definition of $\gamma $, the vertical space $V^{\gamma (t)}$ consists of vectors $J(t)$ with $J\in \mathcal I$.   

For any $t$ in the domain of definition of $\gamma$, the space $\mathcal U$ contains 
the subspace $\mathcal W^{t}$ of all normal Jacobi fields along $\gamma$ satisfying the initial conditions $J(t)=0$ and $J'(t)\in H^{\gamma (t)} \cap \gamma ^{\perp}$, thus given by variations through  horizontal  geodesics containing $\gamma (t)$.  

If  $\gamma  (t)$ is regular (in particular, if $t$ is small enough),  $\mathcal W^t$ has dimension $m-1$ and intersects $\mathcal I$ only in $0$.  For $t=0$, the direct sum of $\mathcal W^0$ and $\mathcal I$ is exactly the space of all normal $L$-Jacobi fields along $\gamma$. 

Let now  $t_0 >0$ be  smaller than the normal injectivity radius of $L$.  Then the space $\mathcal W^0\oplus \mathcal I$  of normal $L$-Jacobi fields does not have focal times in $(0,t_0]$. Hence this space  intersects $\mathcal W^{t_0}$ only in $0$. 

By the definition of the canonical symplectic form on the space of Jacobi fields,  the space  $\mathcal W^t$ is contained  in the  $\omega$-orthogonal  space $\mathcal I^{\perp _{\omega}}$ of $\mathcal I$, for any $t$ in the domain of definition of $\gamma$.   Hence,
$$\mathcal W^0\oplus \mathcal I \oplus \mathcal W^{t_0} \subset  \mathcal I^{\perp _{\omega}}  \; \;  \text{and} \; \; \mathcal W^0\oplus \mathcal I \oplus \mathcal W^{t_0} \subset  \mathcal U\;.$$  

Since $\dim (\mathcal W^0\oplus \mathcal I \oplus \mathcal W^{t_0})=\dim (\mathcal U)=\dim (\mathcal I^{\perp _{\omega}} )$,  both inclusions are equalities. Hence, $\mathcal U= \mathcal I^{\perp _{\omega}}  \,.$

\subsection{Transversal versus basic Jacobi equation}  \label{subsec: varhor}
In the above notation, 
consider the  quasi-geodesic $\bar \gamma =P\circ \gamma$ in $Y$.  Since $\gamma$ is typical,  $\bar \gamma$ is an orbifold-geodesics  in the smooth Riemannian orbifold $Y^m\cup Y^{m-1}$. 
 All subsequent considerations are local, hence, we may replace $Y$ by $Y^m\cup Y^{m-1}$ and  $M$ by the correpsonding preimage. Thus, we may   assume  
$Y=Y^m\cup Y^{m-1}$.

As verified in Section \ref{subsec: jacident},  the $\omega$-orthogonal $(n+m-2)$-dimensional space $\mathcal I^{\perp _{\omega}}$ of normal Jacobi fields coincides with $\mathcal U$.

Since horizontal geodesics in $M$ are sent by $P$ to orbifold-geodesics in $Y$,  the differential of $P$ sends any Jacobi field $J\in \mathcal U$ to a normal Jacobi field $\bar J:=DP\circ J$  along $\bar \gamma$.

The map $J\to \bar J$ is linear and has exactly $\mathcal I$ as its kernel. Since the space $\bar {\mathcal U}$ of all normal Jacobi fields along $\bar \gamma$ has dimension $2m-2$, the map $J\to \bar J$ is a surjective map $\mathcal U\to \bar {\mathcal U}$.

Every Jacobi field $\bar J$ in $\bar {\mathcal U}$ is a solution of the Jacobi equation along the orbifold-geodesic $\bar \gamma$ in the smooth Riemannian orbifold $Y$:
\begin{equation} \label{eq: jacc}
 \bar J'' +\bar {\mathcal R} (\bar J)=0\;,
\end{equation}
where $\bar {\mathcal R}$ is the Jacobi operator along $\bar \gamma$ in the normal bundle $\bar \gamma ^{\perp}$.

On the other hand,  we use the $\omega$-isotropic vector space $\mathcal I$ of normal Jacobi fields to define the transversal Jacobi equation along $\gamma$ as recalled in Section 
\ref{subsec: trans} above.   More precisely,  there is a smooth bundle $\tilde H$ along $\gamma$, such that at any regular point $\gamma (t)$, the fiber $\tilde H^{\gamma (t)}$ is the orthogonal complement of the vertical space $V^{\gamma (t)}$ within $\gamma ^{\perp} (t)$.  Thus, $\tilde H^{\gamma(t)}$ is the orthogonal complement of $\gamma '(t)$ in  $H^{\gamma (t)}$.

Then
the orthogonal projection $\tilde J$ of any $J\in \mathcal U$ onto 
 $\tilde H$ solves the transversal Jacobi equation 
\begin{equation} \label{eq: jacc1}
 \tilde \nabla \tilde \nabla \tilde  J +\tilde {\mathcal R} (\tilde  J)=0\;,
\end{equation}
where $\tilde \nabla$ is the covariant derivative    and  $\tilde  {\mathcal R}$ is the transversal Jacobi operator on  $\tilde H$  defined in  \cite{Wilking} and recalled in Section \ref{subsec: trans} above.

The differential $DP$ of $P$ is an isometric identification of the bundle $\tilde H$ along $\gamma$ and the normal bundle  $\bar \gamma^{\perp}$ along $ \gamma$, on the regular part of $M$.  By  continuity, this extends  to the discretely many non-regular points in $\gamma$.
Upon this identification, we obtain on the bundle $\tilde H$ 
two (a priori different) pairs of covariant derivatives and Jacobi operators $\tilde {\mathcal R}$ and $\bar {\mathcal R}$. 

As observed above the $(2m-2)$-dimensional space $\bar  {\mathcal U}$ of sections in this Euclidean bundle $\tilde H$ is exactly the space of  Jacobi fields for \emph{both} pairs  $(\nabla_{\bar \gamma '}, \bar {\mathcal R})$ and
$(\tilde \nabla, \tilde {\mathcal R})$.   However,  this implies that both covariant derivatives coincide and $\bar {\mathcal R} =\tilde {\mathcal R}$ as verified in \cite[Lemma 2.3]{I-L} or can be seen by a direct computation.

\subsection{Transversal equality}
Let $\gamma _1$ and $\gamma _2$ be two typical horizontal geodesics in $M$ with the same projection $\bar \gamma=P\circ \gamma _1 =P\circ \gamma _2$. 
For all regular times $\gamma _i(t)$ we have a canonical identification of the horizontal subbundle $\tilde H_i$ of the normal bundle $\gamma _i^{\perp}$ with 
the normal bundle $\bar \gamma ^{\perp}$. This provides a canonical identification of  the vector bundles  $\tilde H_i ^{\gamma _i (t)}$  for $i=1,2$, and, by continuity, also for non-regular $\gamma _i (t)$.

Due to Section \ref{subsec: varhor}, this identification preserves the transversal covariant derivative and the transversal Jacobi equation.

\subsection{A conclusion}  We  now easily derive the O'Neill formula:
\begin{proof}[Proof of Proposition \ref{prop: oneil}]
Thus, let $Y$ be a Riemannian manifold. Assume first that  $z\in M$ is a typical point.  Then any horizontal geodesic starting at $z$ is a typical geodesic, since all quasi-geodesics in $Y$ are typical.
 Let $E\subset H^z$ be an arbitrary plane spanned by orthogonal unit vector $h$ and $w$.   

 Consider the geodesic $\gamma=\gamma ^h$.  As in Section \ref{subsec: varhor}, we identify the 
transversal Jacobi operator $\tilde {\mathcal R}$  along $\gamma$ with the Jacobi operator along $\bar \gamma$.  Thus, the sectional curvatures $\kappa (E)$ of $E$ and $\bar \kappa (\bar E)$  of the projection $\bar E=DP (E)$ in $T_{P(z)} Y$ satisfy
$$\bar \kappa (\bar E)=\langle \bar {\mathcal R} (D_zPw), D_zP(w)\rangle= \langle \tilde {\mathcal R} (w),w\rangle  =$$
$$=\langle  {\mathcal R} (w),w\rangle  +3\cdot  ||\tilde A ^{\ast}(w)||^2 =\kappa (E) +3\cdot  ||\tilde A ^{\ast}(w)||^2 \geq \kappa (E) \,,$$
by the definition of $\tilde {\mathcal R}$, \eqref{eq: defr}.   Here, the endomorphisms $\tilde A$ along $\gamma$ assign  to  vertical vectors $J(t)$ the horizontal component of $J'(t)$, for any holonomy field  $J\in \mathcal I$, see Section \ref{subsec: trans}.

This directly implies $\bar  \kappa (\bar E) \geq \kappa (E)$, for all horizontal planes $E$ at all typical points $z \in  M$.  By continuity, we deduce \eqref{eq: veryfirst}, thus the inequality  
$\kappa (E) \leq \bar \kappa (\bar E)$, for all horizontal planes $E$ at all $z\in M$. 

Assume now that equality in \eqref{eq: veryfirst}  holds  for all  horizontal planes.
Then, along any typical horizontal geodesic $\gamma$ the field of endomorphisms $\tilde A^{\ast}$ has to be constantly $0$, hence $\tilde A=0$ as well.  By the definition of $\tilde A$, this implies that the vertical bundle is parallel along any such geodesic. Thus, the horizontal bundle is parallel along any typical horizontal geodesic as well.

By continuity, we deduce that the horizontal bundle is parallel along any horizontal geodesic.  By the torsion-freeness of the Levi-Civita connection, for all Lipschitz continuous horizontal vector fields $h_1,h_2$ on $M$ the almost everywhere defined Lie bracket $[h_1,h_2] $ is almost everywhere horizontal.    The Frobenius theorem (in case of Lipschitz distributions,  see \cite{Rampazzo}) implies that every $x\in M$ is contained in an $m$-dimensional $\mathcal C^{1,1}$ submanifold $C_x$ of $M$, \emph{the section through $x$} whose tangent space at every $p\in C_x$ is the horizontal space $T_pC_x=H^p$.    

The restriction $P:C_x\to Y$ is  locally isometric with respect to the intrinsic metric of $C_x$.  Since $P$ is $1$-Lipschitz, $C_x$ must be locally convex, hence totally geodesics.  Thus, $P:C_x\to Y$ is an isometry in a neighborhood of $x$.

 If, on the other hand, every $x\in M$ is contained in a totally geodeisc submanifold $C_x$, which is mapped by $P$ isometrically onto a neighborhood of $y=P(x)$ in $Y$, then all tangent spaces of $C_x$ need to be horizontal.  Moreover, for every plane $E\subset T_xC_x=H^x$, the curvature of $E$ in $M$ equals the curvature of $E$ in $C_x$, hence, by assumption, it also equals  the curvature  $\bar \kappa (D_xP(E))$.  This finishes the proof.
\end{proof}

\newpage
\centerline{ \bf III. Smoothness without curvature assumptions}

\section{Non-transnormality implies non-smoothness} \label{sec: nontrans}
 
\begin{thm} \label{thm: transnormal}
Let $M$ be smooth and $P:M\to Y$ be a local submetry. If all regular fibers of $P$ are $\mathcal C^2$-submanifolds 
then $P$ is transnormal.
\end{thm}

\begin{proof}
Assume the contrary. Thus,  there is a fiber $L$ of $P$ 
 which is not a $\mathcal C^{1,1}$-submanifold.     
Due to  \cite[Theorem 1]{L-note}, we find a point $x\in L$ such that $T_xL= V^x$
is a $k$-dimensional halfspace and such that $L$ is $k$-dimensional in a neighborhood of
$x$.

Denote by $\hat V$ the $k$-dimensional linear subspace of $T_xM$ spanned by   $V^x$.
 Let  $V_0$ be the  hyperplane in  $\hat V$
which bounds the halfspace $V^x$.
Denote by $v_0$ the unique unit vector in $V^x$ which is orthogonal to $V_0$.

We find  $\varepsilon >0$ and  a $\mathcal C^{1,1}$-curve $\gamma :[0,\varepsilon]\to L$,   
with  $\gamma '(0)=v_0$.
(The existence of such $\gamma$ is obvious if $L$ is a convex subset. The general case follows from  the convex case by  
 \cite[Theorem 2]{L-note}).    

Since $\gamma $ is $\mathcal C^{1,1}$, we find   a $\mathcal C^2$-curve $\tilde \gamma:[0,\varepsilon) \to M$, with   $\gamma (s_i)=\tilde \gamma (s_i)$ for some $s_i> 0$ converging to $0$.
In particular, $\tilde \gamma (0)=x$ and $\tilde \gamma  '(0)=v_0$.


 The horizontal space $H^x$  is the dual cone of $V^x$.  Thus, $H^x$ is the set of all vectors $h$ in the orthogonal complement $V_0^{\perp}$ such that $\langle h,v_0 \rangle \leq 0$.  This  is a halfspace of   $V_0^{\perp}$  bounded by the hyperplane $H_0 =(\hat V)^{\perp}$.

The differential $D_xP:T_xM\to T_yY$ and its restrictions 
 $D_xP:H^x \to T_yY$ and $D_xP:H_0 \to T_yY$ are submetries  \cite[Proposition 5.5]{KL}.

Consider an arbitrary  regular unit vector $w\in T_yY$.  Denote by $Q$ its fiber $D_xP ^{-1} (w) \subset H^x$ and
by $Q_0$ the intersection $Q_0=Q\cap H_0$.   Let $h\in Q_0$ be arbitrary.  Set $z_{\delta}:= \exp _x (\delta \cdot h) \in M$ and  denote by $L^{\delta}$ the fiber of $P$ through $z_{\delta}$.

   We will proceed along the following lines.  We  find a  $\mathcal C^{1,1}$-curve $\eta:(-\varepsilon, 0]\to Q$,  which meets  $H_0$ orthogonally at $h=\eta (0)$.   Exponentiating $\eta$ we obtain 
$\mathcal C^{1,1}$ curves $\eta _{\delta}$ in $L^{\delta}$, contained in the $\delta$-sphere around $x$.   By the assumption, $L^{\delta}$ is $\mathcal C^2$. Thus,  we find a 
$\mathcal C^2$ curve $\tilde \eta _{\delta}$, which intersects $\eta _{\delta}$  at a sequence of times around $0$.  As $\eta _{\delta}$  is contained in the $\delta$-sphere around $x$, this curve  $\tilde \eta _{\delta}$ has very large geodesic curvature at the origin.  Then the second variation formula will imply that  the distance from $\tilde \eta _{\delta} (s_i)$ and $\tilde \gamma (s_i) =\gamma (s_i)$ is less than $\delta$, in contradiction to the  fact that no point of $L_{\delta}$ has distance less than $\delta$ to $L$.

First,  we claim that $Q$ is a $\mathcal C^{1,1}$-manifold with a  non-empty  boundary $Q_0$,  such that $Q$  meets the hyperplane  $H_0$ orthogonally. 

Indeed, consider the composition  $F= D_xP \circ I :V_0^{\perp}\to T_yY$ of the submetry  $D_xP:H^x\to T_yY$ and the quotient map $I:V_0^{\perp} \to H^x$ for the reflection at the hyperplane $H_0 \subset V_0 ^{\perp}$.  Then  $F$ is a submetry and $\hat Q:=F^{-1} (w)$ is a regular fiber of  $F$. Hence, $\hat Q$   is a $\mathcal C^{1,1}$-submanifold.  The submanifold
$\hat Q$ has dimension one more than the dimension of $Q_0$ and $\hat Q$ is invariant under the reflection at $H_0$.  This implies the claim.

 Since  $Q \subset H^x$ meets the hyperplane $H_0$
orthogonally, we find   a $\mathcal C^{1,1}$-curve $\eta:(-\varepsilon, 0] \to Q$ 
with  $\eta (0)= h$ and $\eta '(0)= v_0$.   

 Due to \cite[Corollary 6.3]{L}, the point 
$\exp _y(\delta \cdot w)$ is regular point of $Y$, for all small $\delta >0$.  Hence, $L^{\delta}$ is a regular fiber of $P$. The distance from $L$ to  any  point in $L^{\delta}$, sufficiently close to $z$, is $\delta$.

The curve $\eta _{\delta} (t):=\exp _x (\delta \cdot \eta (\delta ^{-1} \cdot t ))$ is  $\mathcal C^{1,1}$ with $\eta_{\delta} (0)=z_{\delta}$.   Once $\delta$ is small enough, $\eta _{\delta}$ is contained in $L^{\delta}$, \cite[Proposition 7.3]{KL}.



By assumption, the regular fiber  $L^{\delta}$ is a $\mathcal C^{2}$-submanifold of $M$.
Thus, we find a $\mathcal C^2$-curve $\tilde \eta _{\delta} :(-\varepsilon, \varepsilon)\to L^{\delta}$ which
coincides with $\eta _{\delta}$ at  a sequence of negative times   $t_i$ converging to $0$.
 By construction, 
$$\tilde \eta _{\delta}(0)=z_{\delta}\;\; ;\;\;  \tilde \eta _{\delta}  '(0)  = \eta _{\delta}'(0)  =\delta^{-1} \cdot  J(\delta)\,,$$
where $J$ is the Jacobi field along  the geodesic $c(t)=\exp _x(t\cdot h)$ with the initial conditions $J(0)=0$ and  $J'(0)= v_0$.

Denote   the geodesic curvature  of $\tilde \eta _{\delta} $  in the  direction of $c'(\delta)$ by  $a_{\delta}$.   Since $\tilde \eta _{\delta} (t_i)$ is contained in the $\delta$-sphere around $x$, we have
$$a_{\delta}:= \langle \nabla _{\tilde \eta _{\delta}} \tilde \eta_{\delta} , c'(\delta)  \rangle  \geq  (2\delta)^{-1}\,,$$
for all $\delta$ small enough.

Denote the geodesic curvature  of $\tilde \gamma$ in the  direction  $h=c'(0)$ by
 $$a:= \langle \nabla _{\tilde \gamma} \tilde \gamma , h  \rangle \,.$$  
%

Set $l_{\delta}(s)= d(\tilde \gamma (s), \tilde \eta (s))$.  Then $l_{\delta}(0)=\delta$ and $l_{\delta}'(0)=0$, by the first variation formula.  We claim that $l_{\delta}''(0)<0$, for $\delta$ small enough.


Indeed, by  the second variation formula,  \cite[Proposition 4.7]{Ballmann}:
$$l_{\delta}''(0)= a-a_{\delta} +  \mathcal M _{\delta}\,,$$
where $M_{\delta}$ denotes the \emph{index form}
$$M_{\delta}:=\int _{0} ^{\delta} (||J _{\delta} '||^2 - \langle \mathcal R (J_{\delta}), J_{\delta}\rangle)\;.$$
Here $\mathcal R$ is  the Jacobi operator and   $J_\delta$ is  the  Jacobi field  along the geodesic $c$  with $J_{\delta}(0)=\tilde \gamma'(0)= v_0$   and $J_{\delta}(\delta)=\tilde \eta '(0)=\delta ^{-1} \cdot J(\delta)$.  

After rescaling, we may assume that the norm of  all sectional curvatures in a small ball around $x$ is bounded from above by $1$.
 Then, for all  small $\delta$ the Rauch comparison implies, \cite[Lemma 2.1]{Jacobi}
$$|J_{\delta}(\delta)- J_{\delta} (0)|=|\delta ^{-1} \cdot J(\delta)-v_0| \leq || J'(0)|| \cdot \delta=\delta\,. $$
Applying the Taylor formula twice we deduce that $J'_{\delta}$ is uniformly bounded on $[0,\delta]$. Hence, $M_{\delta}$ is uniformly bounded for all sufficiently small $\delta$. 

Since $a$ is independent of $\delta$ and   $a_{\delta}$ converges to $-\infty$ for $\delta \to 0$, we  have verified  that $l_{\delta}''(0)<0$, for $\delta$ sufficiently small.  Fix any such $\delta$.  Then, for all sufficiently small positive $s$
$$l_{\delta} (s)= d(\tilde \gamma (s), \tilde \eta (s)) <\delta\;.$$

Hence for all sufficiently large $i$, the  distance between the points $\tilde \eta (s_i)$ and
 $\tilde \gamma (s_i)=\gamma (s_i)\in L$ is smaller than $\delta$.  

This contradiction finishes the proof.
\end{proof}

\section{Continuation of smoothness} \label{sec: cont}


\begin{lem} \label{lem: regsm}
Let $P:M\to Y$ be a Riemannian submersion between smooth Riemannian manifolds.  Let
$Q$ be a dense subset of points in $Y$ such that $P^{-1} (y)$ is $\mathcal C^{k,\alpha}$ 
for all $y\in Q$.  Then $P$ is  a $\mathcal C^{k,\alpha}$-map.
\end{lem}

\begin{proof}
The statement is local. We fix some $x\in M$ and a sufficiently small ball $B
=B_{r}(x)$. Set $y= P(x)$. We find points $y_i\in Q\cap P(B)$ for $1\leq i\leq m$, sufficiently close to $y$, so that  $f_i=d_{y_i}$  define distance coordinates in a neighborhood $B'$ of $y$.

By assumption, the fibers $L_i:= P^{-1} (y_i)$ are  $\mathcal C^{k,\alpha}$, hence the distance functions $d_{L_i}$ are  $\mathcal C^{k,\alpha}$ on $B\setminus L_i$, Section \ref{Sec: 8}. 
 In the ball $B$ we have the equality   $f_i \circ P =d_{L_i}$ , \cite[Corollary 2.9]{KL}. Hence
$P$ is $\mathcal C^{k,\alpha}$ in a neighborhood of $x$.
\end{proof}

The proof shows that, if $M$  is complete, it is enough for the conclusion   of Lemma \ref{lem: regsm} to assume that $m=\dim (Y)$ fibers in  a sufficiently  general position are $\mathcal C^{k,\alpha}$.

We will restrict from now on to the case $\alpha =0$, as it will be sufficient for our needs. 
By an abuse of notation, we say that a local  submetry $P:M\to Y$ is $\mathcal C^{k}$ on an open subset $U$ of $M$, if  $L\cap U$ is a $\mathcal C^{k}$ submanifold of $M$, for any fiber $L$ of $P$.  

  If $Y=Y_{reg}$, thus if  $P$ is a Riemannian submersion, the local submetry $P$ is $\mathcal C^{k}$ in this sense if and only if the map $P$
 is a $\mathcal C^{k}$-map   in the usual sense,  due to Lemma \ref{lem: regsm}.

The local submetry $P$ is $\mathcal C^1$ on $U\subset M$  if and only if the restriction of $P$ to $U$ is transnormal. 

\begin{prop} \label{thm: ext}
Let $P:M\to Y$ be a transnormal local submetry. Let $P$ be $\mathcal C^{k}$ on an open subset $U\subset M$.  Let $x\in U$ be arbitrary and let $\gamma:[0,t]\to M$ be a horizontal geodesic starting in $x$.  Then  a neighborhood of $z=\gamma (t)$ in the fiber $L^z$ through $z$ is $\mathcal C^{k}$.  
\end{prop}

\begin{proof}
We proceed by induction on $k$. The case $k=1$ holds true by transnormality.  Thus, we  may assume $k\geq 2$ and that a neighborhood $V$ of $z$ in $L^z$ is $\mathcal C^{k-1}$.

Let $w=-\gamma'(t)   \in H^z$ denote the horizontal incoming direction of $\gamma$.  We approximate $w$ by horizontal directions $w'\in H^z$ such that the corresponding horizontal  geodesics $ \gamma^{w'}$ are regular.  For $w'$ sufficiently close to $w$, the geodesic $\gamma ^{w'}$ intersects the open set $U$.   Since all but discretly many points 
on $\gamma ^{w'}$ are regular, we may replace $x$ by a point on $\gamma^{w'}$ and assume that $x$ is a regular point.

We claim that $w$ extends to a $\mathcal C^{k-1}$  normal field in a neighborhood $V'$ of $z$ in $L^z$.  Once this is done, we can apply  the same argument to any normal vector $w' \in H^z$ close enough to $w$, so that $\exp _z (tw')$ is a regular point in $U$.     Then, choosing a basis $w_1,...,w_r$ of $H^z$ consisting of vectors close enough to $w$, we infer  that the normal bundle of $L^z$ is $\mathcal C^{k-1}$ in a neighborhood $V''$ of $z$ in $L^z$.    Applying Theorem \ref{thm: normalfields}, we then conclude that $V''$ is a $\mathcal C^{k}$-submanifold of $M$.

 It remains to verify the claim.  Consider the basic normal field  $\hat h$ along $L^x$
 with $\hat h_x= \gamma'(0)$.
  By  Lemma \ref{lem: regsm},  $\hat h$ is  
 $\mathcal C^{k-1}$ on $U\cap L^x$. 

 Consider the set   $\mathcal D^t$  of all points $p\in L^z$, such that 
$$Hol^{\gamma} (p) := \exp _p (t\hat h_p)$$  
is defined, see  Section \ref{subsec: holmap2}.  This holonomy  map $Hol^{\gamma}:\mathcal D^t\to L^z$
is Lipschitz-open by  \cite[Proposition 7.2]{L}. 
Therefore,  $Hol^{\gamma}$ is a $\mathcal C^{k-1}$-submersion.   

Hence, we find a $\mathcal C^{k-1}$ submanifold $N$ of $\mathcal D^t$,  which is sent by $Hol ^{\gamma}$  diffeomorphically onto a neighborhood  $V$ of $z$ in $L^z$.      Now we can defined the normal field $\tilde  w$ along $V$ as the following  composition.  

We send a point $q\in V $ by $(Hol^{\gamma})^{-1} (q)$ to a point $p\in N$.  Then we send  $p$ to   the normal vector $\hat h_p \in H^p$ and apply the geodesic flow for the time $t$ to obtain the  horizontal  vector $(\gamma ^{\hat h_p}  ) '(t) \in H^q$.  Finally, we set
 $\tilde w (q):=   -(\gamma ^{\hat h_p}  ) '$.
   By construction, this map $q\to \tilde w (q)$ is $\mathcal C^{k-1}$. This finishes the proof of the claim and of the Proposition. 
\end{proof}

This proposition \emph{does  not imply} that  $P$ is $\mathcal C^{k}$ near 
$z$, thus that \emph{all} fibers near $z$ are $\mathcal C^{k}$.  However, we have some partial results in this direction, see  also  Theorem  \ref{cor: smimpltr}  and Theorem \ref{thm: dualleaf} below.

\begin{prop} \label{prop: spec}
Under the assumption of Theorem \ref{thm: ext}, let $S$ be the stratum of $P$ through  $z$.  Then, there exists a neighborhood $O$ of $z$, 
such that for  any $z'\in S$, the intersection $ L^{z'} \cap O$ is $\mathcal C^{k}$.
\end{prop}

\begin{proof}
As in   the proof of Theorem \ref{thm: ext}, we can
assume that $x$ is a regular point.  Using the notations of the proof of Theorem \ref{thm: ext}, we find a sufficiently small neighborhood $\tilde O$ of $w\in H^z$ in the tangent bundle $TM$, such that $\exp (tw') $ is a regular point in $U$, for all $w'\in \tilde O$.

Along the stratum $S$, the horizontal spaces define a continuous distribution. Thus, we find  some neighborhood $O$ of $z$ in $M$, such that  for  any $z'\in O \cap S$ there exists a horizontal vector $w'\in H^{z'} \cap \tilde O$. 

Now,  Theorem \ref{thm: ext} implies that a neighborhood of $z'$ in $L^{z'}$ 
is $\mathcal C^{k}$.    Since this applies to all $z'\in O\cap S$, we deduce that $L^{z'} \cap O$ is $\mathcal C^{k}$ submanifold of $M$.
\end{proof}

If  $z$ is a regular point of $M$, the stratum $S$ of $z$  is open. Hence
 Proposition \ref{prop: spec} and Lemma \ref{lem: regsm} imply:

\begin{cor} \label{cor: extregfib}
Under the assumption of Proposition \ref{thm: ext},  let  $P(z)$ be  a regular point of $Y$.  Then $P$ is $\mathcal C^{k}$ in a neighborhood of $z$ in $M$.
\end{cor}

We expect that the next statement holds true for all strata.

\begin{prop} \label{prop: extreg1fib}
Under the assumption of Proposition \ref{thm: ext}, assume in addition, that $P(z)$ is contained in the codimension-one stratum $Y^{m-1}$ of  $Y$. Then $P$ is $\mathcal C^{k}$ in a neighborhood of $z$ in $M$. Moreover, the stratum $S$ through $z$ is a $\mathcal C^{k}$-submanifold at $z$.
\end{prop}

\begin{proof}
We choose a sufficiently small open ball $O$ around $z$, such that the conclusions of Proposition \ref{prop: spec} hold true.  We replace $M$ by $O$ and $Y$ by $P(O)$.  By Theorem \ref{thm-sm4}, $Y$ is 
 a smooth Riemannian manifold with totally geodesic boundary $\partial Y = P(S)$.  

Set $y=P(z)$.  By the choice of $O=M$, all non-regular fibers in $M$ are contained in $S$ and are $\mathcal C^{k}$.   

Let $x$ be an arbitrary regular point in $M$ sufficiently close to $z$. Set $y':=P(x)\in Y_{reg}$. The incoming directions in $y'$ of geodesics connectings $y'$ with points on $\partial Y$
generate the tangent space $T_{y'} Y$.  Hence, we find $m$ points $y_i\in \partial Y$ close to
$y$, such that the distance functions $d_{y_i}$ define smooth coordinates 
 in $Y$ around $y'$.  

Around the point $x$, the pull-backs $d_{y_i} \circ  P$ are  given by distance functions $d_{L_i}$, where $L_i$ is the fiber $P^{-1} (y_i)$.  By Proposition \ref{prop: spec}
and Section \ref{Sec: 8}, these distance functions are $\mathcal C^{k}$ in a neighborhood of $x$.

 The gradients of the functions $d_{y_i}$ are linearly independent at $y'$. Hence,  the gradients of $d_{L_i}$ are linearly independent at $x$. Thus, the fibers of the distance functions $d_{L_i}$ intersect transversally at $x$ and their intersection is a $\mathcal C^{k}$ submanifold around $x$.  By construction, this intersection is exactly the fiber $L^x$ of $P$.  

Thus, all fibers of $P$ are $\mathcal C^{k,\alpha}$ in a neighborhood $O$ of $z$.
Hence, $P:O\setminus S\to Y^m$ is $\mathcal C^{k}$.

The distance function  $d_{\partial Y}$ in the smooth Riemannian manifold $Y$ is smooth on a neighborhood of $y$. Hence, the composition $d_S=d_{\partial Y} \circ P$ is $\mathcal C^{k}$ on  $O'\setminus S$ for a small neighborhood  for a neighborhood $O'\subset O$ of $z$.
Hence, $S$  is $\mathcal C^{k}$ submanifold in a neighborhood $O'$ of $z$.
\end{proof}


We now complete the proof of a local generalization of Theorem \ref{thm:last}:

\begin{thm} \label{cor: smimpltr}
Let $P:M\to Y$ be a local submetry and $k\geq 2$.  If all regular fibers are $\mathcal C^{k}$ then $P$ is transnormal and all fibers are $\mathcal C^{k}$.

For any $l$, the stratum  $S^l =P^{-1} (Y^l)$  and  the restriction
$P: S\to Y^l$ are $\mathcal C^{k-1}$.  If $l\geq \dim (Y)-1$ then $S=P^{-1} (Y^l)$  and the restriction
$P: S^l\to Y^l$ are $\mathcal C^{k}$.
%
\end{thm}

\begin{proof}
The transnormality has been verified in Theorem \ref{thm: transnormal}.    

Let now $z\in M$ be arbitrary. Since $z$ lies on some  horizontal geodesic containing a regular point, the fiber $L^z$ is $\mathcal C^k$ by Proposition \ref{thm: ext}.

  The stratum $S^l$ for $l=\dim (M)$ is the open set of regular points in $M$. 
The restiction $P:S^l\to Y^l$ is $\mathcal C^k$ by Lemma \ref{lem: regsm}.

If $l=\dim (M)-1$, then $S^l$ is $\mathcal C^k$ by Proposition \ref{prop: extreg1fib}. 

Let now $l$ be arbitrary. Let $x\in S$ be arbitrary and let $L$ be the fiber through $x$.  Fix  any  sufficiently small vector $h\in H^x_0$ (the normal space of $L$ in $S$). Then   the exponential image of the basic normal field $\hat h$ extending $h$ to a neighborhood of $x$ in $L$ is an open subset in 
another fiber of $P$ contained in $S$, hence a $\mathcal C^k$ submanifold of $M$.  Since the normal exponential map of $L$ is the restriction of a smooth map to a $\mathcal C^{k-1}$ submanifold, we deduce that $\hat h$ is $\mathcal C^{k-1}$.  Hence the normal bundle of $L$ in $S$ is a $\mathcal C^{k-1}$ submanifold of the tangent bundle of $M$.  
Applying the  exponential map, we deduce that a neighborhood of $z$
in $S$ is $\mathcal C^{k-1}$.

  The statement that the Riemannian submersion $P:S^l \to Y^l$ is $\mathcal C^k$  for $k=l-1$ and $\mathcal C^{k-1}$ for all $l$, follows exactly as in Lemma \ref{lem: regsm}.
\end{proof}


\section{Focal times and focalizing Jacobi fields} \label{sec: focal}

\subsection{Focal points, vertical and horizontal}
Let $P:M\to Y$ be a transnormal  local submetry.  Let $L$ be a regular fiber of $P$ and  $x\in L$ be a \emph{good}  point.
Let $h\in H^x$ be a unit normal vector and $\gamma =\gamma ^h$ be the geodesic in the direction of $h$.

The Lagrangian vector space 
$\mathcal W =\mathcal W(h)$ of  $L$-Jacobi fields along $\gamma$ is well-defined and so is 
its $\dim (L)$-dimensional subspace $\mathcal I=\mathcal I^h$ of holonomy fields, see Section \ref{subsec: goodpoints}.

As in Section \ref{subsec: focalpoints}, denote for any $t$ in the domain of definition of $\gamma$, by  $\mathcal W^t$ the space  of elements $J\in \mathcal W$ with $J(t)=0$.    As before, we call $t$ a focal time (along $\gamma$) if $\mathcal W^t$ is non-trivial. 

Following, \cite{Jacobi}, \cite{Thorb}, we call the number $\dim (\mathcal W^t \cap \mathcal I)$ the \emph{vertical focal multiplicity} (of $\gamma$, or of $L$ along $\gamma$).
 By Lemma \ref{lem: consequent},  the holonomy fields span at the time $t$ exactly the vertical space
$V^{\gamma (t)}$.  Hence, the vertical focal multiplicity of the focal time $t$, is exactly the difference of the dimensions 
$$\dim (V^{\gamma (0)})-\dim (V^{\gamma (t)}) =
n-m-\dim (V^{\gamma (t)})\;.$$
Thus, the vertical focal multiplicity is positive if and  only if  $\dim (V^{\gamma (t)})$ (hence the dimension of the fiber of $P$ at $\gamma (t)$)  is not equal $\dim (L)=n-m$.   In such a case, we say that $t$ is a \emph{vertical focal time}.

As the \emph{horizontal focal multiplicity} of the time $t$, we denote $$\dim (\mathcal W^t) -\dim (\mathcal W^t \cap \mathcal I) \;.$$
   As observed in  \cite[Lemma 3.1]{Jacobi}, the horizontal focal multiplicity of the time $t$  is exactly the $\tilde {\mathcal W}$-focal multiplicity of $t$ for the transversal Jacobi equation, where $\tilde {\mathcal W}=\mathcal W /\mathcal I$, see Section \ref{subsec: trans}.   If this horizontal focal multiplicity is positive, we call $t$ a \emph{horizontal focal time}.

\subsection{The case of typical geodesics}  
Assume in addition, that $L$ is a typical fiber, $x$ a typical point and $\gamma$ a typical geodesic.  Then, by definition, 
$P\circ \gamma$  is contained in the Riemannian orbifold part $Y^m\cup Y^{m-1}$.  
Moreover, all conjugate points of $y=P(x)$ along the orbifold-geodesic $\bar \gamma =P\circ \gamma$
are not contained in $Y^{m-1}$.  

The identification of the transversal Jacobi equation along $\gamma$ with the Riemannian Jacobi equation along $\bar \gamma$, explained in Section \ref{subsec: varhor},  implies that  the horizontal focal  multiplicities of $\gamma$ coincide with the focal (=conjugate)  multiplicities of the point $\bar \gamma (t)$ along $\bar \gamma$ with respect to $y=\bar \gamma (0)$ in the Riemannian orbifold $Y^m\cup Y^{m-1}$.

In particular,  vertical  and horizontal focal times of $L$ along a typical geodesic $\gamma$ never coincide.

\subsection{Spaces of focalizing Jacobi fields}  Here comes the key result, which will enable   boot-strapping to smoothness:

\begin{prop} \label{jac: family}
Let $L$ be a typical fiber of  a transnormal local submetry $P:M\to Y$. Let 
$\hat h$  be a typical basic normal field along $L$. Let $t\in \mathbb R$ be fixed  and  $\mathcal D^t$ be the set of points $q\in L$ for which the horizontal geodesic $\gamma _q:=\gamma ^{\hat h_q}$ in the direction of $\hat h_q$ is defined at 
time $t$.

For $q\in \mathcal D^t$, denote by $\mathcal W_q ^t$ the space of all normal Jacobi fields $J$ along $\gamma _q$ such that $J(t)=0$, $J'(t)$ is  horizontal and $J(0)\in V^q$. Then  
\begin{enumerate}
\item If $q$ is a typical point then $\mathcal W_q^t=\mathcal W^{t}(\hat h_q)$, the space of all $L$-Jacobi fields along $\gamma _q$, which focalize at time $t$.
\item The map $q\to \mathcal W_q^t$ is locally Lipschitz. In particular, $\dim (\mathcal W_q^t)$ is constant on connected components of $\mathcal D^t$.
\item The map $q\to \mathcal W_q^t$ is $\mathcal C^{k-1}$ if  $P$ is  $\mathcal C^k$ in a neighborhood of $\mathcal D^t$.
\end{enumerate}

\end{prop}

\begin{proof}
The statement is local, so we may assume that $\mathcal D:=\mathcal D^t$ is connected.
For $q\in \mathcal D$, let $\tilde  q$ denote the point $\gamma _q(t)$.  Then  $q\to \tilde  q$ is a holonomy map. Hence, its image $\hat {\mathcal D}$ is a connected open subset of a fiber $\tilde  L$ of $P$. By connectedness, the dimension of the vertical space $V^{\tilde q}$
does not depend on $q$.

Let $q\in \mathcal D$ be a typical point.  Any $L$-Jacobi field along $\gamma _q$ satisfies $J(0)\in V^q$.  If, in addition,  $J(t)=0$, then $J'(t)$ must be horizontal: Indeed,  $J'(t)$ must be orthogonal to  $\hat J (t)$, for all holonomy fields  $\hat J \in \mathcal I^{\hat h_q}$, as explained in Section \ref{subsec: jacident}, and the set of all such values $\hat J(t)$ is exactly $V^{\tilde q}$.

On the other hand, any normal Jacobi field $J\in \mathcal W^t_q$ satisfies $J(t)=0$ and $J'(t)\in H^{\tilde  q}$.  Thus, $J$ is a variational field of a variation through horizontal geodesics passing through $\tilde q$. Since $J(0)\in V^q$, we find a holonomy field $\hat J$, such that $\hat J(0)=J(0)$. Then $J-\hat J$ is a variational field of a variation through horizontal geodesics (Section \ref{subsec: jacident}), which vanishes at $0$. Hence $J-\hat J$ is an $L$-Jacobi field and so is $J$.   This finishes the proof of (1).


 Let $\bar \gamma$ denote the projection $\bar \gamma =P\circ \gamma _q$, which is independent of $q \in \mathcal D$. 
For $q\in \mathcal D$ we  consider the vector space   
$$K^q= \{J'(t) \in H^{\tilde q} \; \;  , \,\;  J\in \mathcal W_q^t\}\,.$$
Since Jacobi fields depend smoothly on the pair $(J(t),J'(t))$ and $J(t)=0$ for all $J\in  \mathcal W_q^t$, we  need to verify that the vector space $K^q\subset H^{\tilde  q}\subset  T_{\tilde q} M$ depends in a Lipschitz way on $q$ and that this dependence is $\mathcal C^{k-1}$, if $P$ is $\mathcal C^k$ near $\mathcal D$.

  Assume  that $\tilde  L$ is a regular fiber of $P$.  Set $\tilde y =P(\tilde L)$ and $y=P(L)$.  We fix $q\in \mathcal D$. For any $w\in K^q$, the unique Jacobi field $J^w\in \mathcal W_q^t$ with $(J^w)'(t)=w$   is the variational field of a variation of $\gamma _q$ by horizontal geodesics passing through $\tilde q$.  Composing with $P$ we deduce that $\bar J^w:=DP (J^w)$ is a Jacobi field of a variation of $\bar \gamma$ by orbifold-geodesic passing through $\tilde y$.    We have $\bar J^w(t)=0, (\bar J^w)'(0)=D_{\tilde q} P (w)$ and $\bar J^w (0)=0$, since 
$J^w(0)$ is vertical.

For $u\in T_{\tilde y} Y$ consider the unique Jacobi field $\bar J^u$ along $\bar \gamma$ with $\bar J^u (t)=0$ and $(\bar J^u )'(t)=u$.  Denote by $K$  the space of all vectors $u \in T_{\tilde y} Y$, such that $J^u (0)=0$.  As we have seen above, the differential 
$D_{\tilde q} P$ sends $K^q$ to a subpace of $K$.

On the other hand, for any $u\in K$, consider the corresponding Jacobi field $\bar J^u$ and
the unique horizontal vector $u_{\tilde q} \in H^{\tilde q}$ which is sent by $D_{\tilde q} P$ to $u$.
Consider the Jacobi field $J$ along $\gamma _q$ given by $J(t)=0$ and $J'(t)=u_q$.
The field $J$ can be obtained as a variational field of horizontal geodesics passing through $\tilde q$.
Hence the composition $DP (J)$ is a Jacobi field, which must coincide with $\bar J^u$, by the initial coditions.  Hence, $DP(J(0))=0$ and $J(0)$ must be a vertical vector.  Therefore, $J\in \mathcal W^t _q$ and $\tilde u_q \in K^q$.  Hence, the map $D_{\tilde q } P:K^q\to K$ is an isomorphism.

We have verified that for any $q\in \mathcal D$, the isomorphism $D_{\tilde q} P :H^{\tilde q} \to T_{\tilde y} Y$ sends $K^q$  to $K$. Since the map $q\to \tilde q$, $\tilde q\to H^{\tilde q}$ and $\tilde q\to D_{\tilde q} P$ are Lipschitz continuous so is the map $q\to K^q$.

If, in addition, $P$ is $\mathcal C^k$ in a neighborhood of $\mathcal D$, then the map $q\to \tilde q$ is $\mathcal C^{k-1}$. Moreover, $P$ is $\mathcal C^k$  around  $\tilde q$, for any $q\in \mathcal D$, by Corollary   \ref{cor: extregfib}.   Hence, $H^{\tilde q}$ and $D_{\tilde q} P$ depend in $\mathcal C^{k-1}$ way on $\tilde q$.  Thus, $q\to K^q$ is $\mathcal C^{k-1}$.

This finishes the proof of (2) and (3) if $\tilde L$ is a regular fiber.

Assume now that $\tilde  y$ is not a regular point. By our assumption, $\tilde y\in Y^{m-1}$ and  $\tilde y$ is not conjugate to $y$ along $\bar \gamma$.

 For any $q\in \mathcal D$, any $J\in \mathcal W^t_q$ is defined as a variational field of variation through horizontal geodesics passing through $\tilde q$. Hence 
$ \bar J:= DP(J)$ is a Jacobi field along the orbifold geodesic $\bar \gamma$. By assumption, $\bar J(t)=0$ and $\bar J(0)=0$.  Since $\tilde y$ is not conjugate to $y$ we deduce $\bar J=0$.

Thus, for a normal Jacobi field $J$ along $\gamma _q$ with $J(t)=0$ and $J'(t) \in H^{\tilde q}$  we have the equivalence
$$J\in  \mathcal W^t_q  \Leftrightarrow  \bar J=0   \Leftrightarrow \bar J'(t)=0 $$

Denote by $w_q\in H^{\tilde q}$ the outgoing direction $w_q:=\gamma_q' (t)$.
Any normal Jacobi field $J$ along $\gamma _q$  with $J(t)=0$ and  $u:=J'(t)\in H^{\tilde q}$ 
is defined by (taking the variational field of the   exponentiation of) a  $\mathcal C^1$-curve
$\eta$  in the unit sphere in $H^{\tilde q}$, such that $\eta (0)=w_q$ and $\eta'(0)=u$.
For such a field $J$,  and its projection $\bar J=DP(J)$, the value $\bar J'(t)$ is given 
by the starting direction of the curve  $D_{\tilde q} P \circ \eta$.

Thus, $\bar J'(t)=0$ if and only if $u$ is a vertical direction at the point $w_q$ for the submetry $D_{\tilde q} P:H^{\tilde q} \to T_{\tilde y} Y$.   

Denoting by 
$F_q$ the fiber through $w_q$ of the submetry $D_{\tilde q} P:H^{\tilde q} \to T_{\tilde y} Y$, we thus have shown
$$K^q=T_{w_q} F_q\;.$$

However, $\tilde y$ is contained in $Y^{m-1}$, $T_{\tilde y}Y= \R^{m-1} \times [0,\infty)$  and we have described the submetry $D_{\tilde q} P$ in details as follows in   Section \ref{subsec: finer}.  The fiber $F_q$ is the   round sphere   of all elements $g\in H^{\tilde q}$ which have the same norm as $w_q$  and the same projection as $w_q$  to the subspace $H^{\tilde q}_0=T_{\tilde q} S$, the tangent space of the stratum $S$ through $\tilde q$.

Hence the tangent space $T_{w_q} F_q$   is the hyperplane $A^{w_q}$ in $H^{\tilde q} _1 =T_{\tilde q}^{\perp} S$, which is orthogonal to $w_q$.

Thus, we have shown $K^q=A^{w_q}$, the hyperplane of all vectors orthogonal to $w_q$  in the normal space $T_{\tilde q} S$ to the stratum $S$.

Since  $w_q$ and $\tilde q$ depend Lipschitz on $q$ and  the normal bundle of $S$ is Lipschitz, we deduce that $K^q$ depend Lipschitz continuously on $q$. 

Moreover, if $P$ is $\mathcal C^k$ around $\mathcal D$, then the maps $q\to \tilde q$ and $q\to w_q$ are $\mathcal C^{k-1}$ and so is the normal bundle of $S$ by 
 Proposition \ref{prop: extreg1fib}.  Thus $q\to K^q$ is $\mathcal C^{k-1}$ as well.
\end{proof}

\section{Smoothness in the Euclidean case} \label{sec: euclid}
We obtain a local generalization of Proposition \ref{thm:  smoothbasic}:

\begin{prop}  \label{prop: meanbas}
Let $P:M\to Y$ be a local submetry. Assume that  for almost every regular fiber $L$, the almost everywhere defined mean curvature vector is a basic normal field. Then $P$  is smooth.
\end{prop}

\begin{proof}
Fix a fiber $L$ as in the assumption.   Any basic normal field along $L$ is Lipschitz, hence, by assumption, so is the mean curvature field.  
Due to  Theorem \ref{thm: normalfields}, any such fiber $L$ is  $\mathcal C^{2,\alpha}$, for  
$\alpha =\frac 1 2$.  Applying Lemma \ref{lem: regsm}, we deduce that $P$ is $\mathcal C^{2,\alpha}$ on the regular part.

Now, we prove by induction on $k$, that $P$ is $\mathcal C^{k,\alpha}$ on the regular part.  Using the inductive assumption, $P$ is  $\mathcal C^{k-1,\alpha}$. Hence, any basic normal  field is $\mathcal C^{k-2,\alpha}$. In particular, for any fiber $L$ as in the formulation of the proposition, the mean curvature is  $\mathcal C^{k-2,\alpha}$.
Due to Theorem 
\ref{thm: normalfields}, $L$ is $\mathcal C^{k,\alpha}$. By Lemma \ref{lem: regsm}, the restriction of $P$ to the regular part is $\mathcal C^{k,\alpha}$.

Hence, all regular fibers of $P$ are smooth.  Due to Corollary  \ref{cor: smimpltr}, $P$ is transnormal and all fibers of $P$ are smooth.
\end{proof}

The following statement is closely related to  \cite[Proposition 4.1]{ARad}, \cite[Propositions 17-18]{Mendes-Rad}.

\begin{thm} \label{thm: eucl}
Let $P:\mathbb R^n \to Y$ be a transnormal submetry.  Then $P$ is smooth and has basic mean curvature.
\end{thm}

\begin{proof}
Due to Proposition \ref{prop: meanbas}, it suffices to prove that any typical  fiber $L$ has  basic mean curvature.   We fix such a fiber $L$.

  Choose a typical basic normal field $\hat h\in \mathcal B(L)$.  Consider, for any $q\in L$ the geodesic $\gamma^{\hat h _q}$.   Due to Proposition \ref{jac: family}, for any pair of typical   points $x,z \in L$, the $L$-focal times and  corresponding multiplicities along $\gamma^{\hat h_x}$ and $\gamma ^{\hat h_z}$ coincide. 
  Since the ambient space is a Euclidean space, the focal times determine the eigenvalues of
the second fundamental form, \cite[Proposition 4.1.8]{PalT}, hence also its trace.
Thus, for all typical $x,z\in L$, the traces of $\mathrm {II} _{L,x} ^{\hat h_x}$ and of $\mathrm {II} _{L,z} ^{\hat h_z}$  coincide.


Thus, for  the  mean curvature vector $\mathfrak h$ along $L$, well-defined at all typical points, the scalar product  $\langle \mathfrak h, \hat h\rangle$ assumes  the same value at all typical points. Since this is true for  any basic normal field $\hat h$, we deduce that the mean curvature vector $\mathfrak h$ is basic.
\end{proof}

\newpage
\centerline{\bf  IV. Smoothness in non-negative curvature }

\medskip
\medskip

\section{$\mathcal C^2$} \label{sec: c2}
If $M$ is complete then any local submetry $P:M\to Y$ is a submetry. 
 Moreover,  the decomposition of $M$ in  connected components of the fibers of $P$ is again a submetry, \cite[Theorem 12.9]{KL}, \cite[Theorem 10.1]{L-subm}  
Thus, for all questions related to smoothness, we can replace the fibers of $P$ by their connected components and assume without loss of generality, that all fibers of $P$ are connected.

\begin{thm} \label{thm: c2neu}
Let $M$ be a complete and  of non-negative curvature. Then any transnormal submetry $P:M\to Y$ is $\mathcal C^2$.
\end{thm}

\begin{proof}
As  explained above, we may assume all fibers of  $P$  to be  connected, \cite[Theorem 12.9]{KL}.
By   Theorem
\ref{cor: smimpltr} and Lemma \ref{lem: regsm}, it suffices   to show that every typical fiber  of $P$ is $\mathcal C^2$.

Let $L$ be a typical fiber of $P$ and let us fix  a typical normal basic field  $\hat h$ along $L$.
For any typical point  $x\in L$ denote by $\mathcal W_x=\mathcal W(\hat h_x)$ the space of $L$-Jacobi fields along $\gamma ^{\hat h _x}$.

We claim that the family $x\to \mathcal W_x$ defined for \emph{all typical} $x\in L$ extends to a \emph{continuous} family of Lagrangians along $\gamma ^{\hat h_x}$ for \emph{all} $x\in L$.

There is a discrete set of times $t_j\in \mathbb R$ (which include the time $0$), such that 
for any typical $x\in L$, the time $t_j$ is a $\mathcal W_x$-focal time with a multiplicity $m_j$ independent of $x$.  Moreover, the $m_j$-dimensional subspace $\mathcal W_x ^{t_j}$ of Jacobi fields in $\mathcal W_x$ focalizing at  time $t_j$ extends in a Lipschitz continuous way to a family $\mathcal W_x ^{t_j}$ defined on all of $L$,  Proposition \ref{jac: family}.

Let us now fix a point $p\in L$ and a sequence of typical  points $p_i\in L$ converging to $p$. Assume, after choosing a subsequence,  that the Lagrangians $\mathcal W_{p_i}$ converge  to a Lagrangian $\hat {\mathcal W}$
of normal Jacobi fields along $\gamma ^{\hat h_p}$.  In order to prove the continuity  claim, it suffices to verify that $\hat {\mathcal W}$ does not depend on the choice of such a sequence  $p_i$. 

 For any focal time $t_j$,  the continuity proved in Proposition \ref{jac: family}, implies
$\mathcal W_p^{t_j} \subset \hat {\mathcal W}$.  Thus, the space $ ^o \mathcal W_p$ generated by all these $\mathcal W_p^{t_j}$ is  a subspace of $\hat {\mathcal W}$, which is  independent of the sequence $p_i$.

By Lemma \ref{lem: indexsemi}, this space $^o{\mathcal W_p}$ contains all focalizing Jacobi fields in the Lagrangian $\hat {\mathcal W}$.  By Proposition \ref{jac: family},  we have $$\hat {\mathcal W}= { ^o \mathcal W} \oplus \hat{\mathcal W}^{par}\,,$$
where $\hat {\mathcal W}^{par}$ consists of parallel Jacobi fields in $\hat {\mathcal W}$.
  Moreover, this decomposition is orthogonal at every $t$.   Therefore,  the space of
 initial values $\{ J(0); J\in \hat{ \mathcal W}^{par}\}$ of parallel vector fields in $\hat {\mathcal W}$ is  exactly the orthogonal complement   of  $\{ J(0); J\in {^o \mathcal W}\}$ inside the space   $\{ J(0); J\in \hat{ \mathcal W}\}$ of all initial values.

 By definition of $\hat {\mathcal W}$ the space of all inital values  $\{ J(0); J\in \hat{ \mathcal W}\}$  is exactly the vertical space $V^p$, thus independent of the sequence $p_i$.    Thus,
 the space of initial values of elements in $\hat {\mathcal W}^{par}$ is also independent of $p_i$. Therefore, $\hat {\mathcal W}^{par}$ is independent of $p_i$ as well.  
 This finishes the proof of the claim.

We  fix a basis $\hat h^1,..., \hat h^m$ of the space of basic normal fields along $L$, such that any $\hat h^e$ is typical.   We apply the previous consideration to all  $\hat h^e$.  The continuous  dependence of the unqiue extension of the space  of $L$-Jacobi fields $\mathcal W(\hat h^e _x)$   on $x$ implies that $L$ is $\mathcal C^2$ by Corollary \ref{cor: contdepend}.
\end{proof}

\section{Rank and smoothness on a dense subset} \label{sec: rank}
Let again $M$ be complete and non-negatively curved and let $P:M\to Y$ be a transnormal submetry. We have already verified that all fibers of $P$ are $\mathcal C^2$.  In particular, any point on any fiber is a good point and any basic normal field is $\mathcal C^1$.

We consider the set of regular points $M_{reg}$  and the bundle of unit horizontal vectors $\mathcal H_{reg}\to M_{reg}$.      For   any vector $0\neq h\in \mathcal H_{reg}$, we  consider 
the  base point $p=\pi (h)\in M_{reg}$  and the fiber $L=L^p=L(h)$ of $P$ through $p$.  Further we consider the geodesic $\gamma ^h$ in the direction of $h$. Along $\gamma ^h$ we have the space $\mathcal W(h)$ of $L$-Jacobi fields.

Let again ${^o\mathcal W}(h)$ be the focal-generated subspace of $\mathcal W(h)$ and 
$ {\mathcal W} ^{par}  (h)$ be the space of parallel Jacobi fields in $\mathcal W (h)$. By Theorem \ref{thm: GAG}, ${\mathcal W} ^{par}  (h)$ and ${^o\mathcal W}(h)$ are complementary and pointwise orthogonal  in $\mathcal W(h)$.

 By the rank of $h$, we denote the number
$$\mathrm{rk} (h):=\dim (\mathcal W ^{par} (h))\;.$$

On $\mathcal H_{reg}$ the second fundamental forms $\mathrm {II} ^h _{L(h)}$ depend continuously on $h$.  Hence, $\mathcal W(h)$ depend continuously on $h$ as well.
Parallel Jacobi fields converge to parallel Jacobi fields under convergence of corresponding geodesics. Thus, for any convergent sequence $h_i \to h$ in $\mathcal H_{reg}$ we have 
$$\limsup (\mathrm{rk} (h_i)) \leq \mathrm{rk} (h)\;.$$
Since the rank can assume only finitely many values, we obtain an open dense subset $\mathcal U$ of $\mathcal H_{reg}$  on which $\mathrm{rk}$ is locally constant.

\begin{thm} \label{thm: openc}
Let $M$ be complete and non-negatively curved. Let $P:M\to Y$ be a transnormal submetry.
Let $\mathcal U$ be the set of unit horizontal vectors at regular points, in whose neighborhoods  the rank function is locally constant.   Then $\pi (U)$ is a dense open subset of $M_{reg}$ and the restriction of $P$ to $U$ 
 is a smooth Riemannian submersion.
\end{thm}

\begin{proof}
The set $\mathcal U$ is by definition open in the  horizontal bundle over $M_{reg}$.  By the semi-continuity of $\mathrm {rk}$, the set $\mathcal U$ is dense.  Since the bundle-projection $\pi$ is open, the image $U:=\pi (\mathcal U)$ is open and dense
in $M_{reg}$.  By Theorem \ref{thm-sm3},  $P(U)$ is a smooth Riemannian manifold and the restriction $P:U\to P(U)$ is a Riemannian submersion.

We will  verify  by induction on $k$ that this Riemannian submersion  $P:U\to P(U)$ is  $\mathcal C^k$.   The case $k=2$ has been proved in Theorem \ref{thm: c2neu}.
Assume the result has already been verified for some $k \geq 2$.   In order  to prove that
$P$ is $\mathcal C^{k+1}$ on $U$, we only need to show  that   the intersection $L\cap U$ is $\mathcal C^{k+1}$, for any \emph{typical} fiber $L$, Lemma  \ref{lem: regsm}.

We fix  a  typical fiber $L$ and a point $x\in L\cap U$. Since $x\in U$,  we find  a  horizontal vector $0\neq h \in \mathcal U \cap H^x$.  Since $L$ is a typical  fiber and  $\mathcal U$ is open, we may further assume that $h$ extends to a \emph{typical basic normal field}  $\hat h$ along $L$.     By inductive assumption, $\hat h$ is $\mathcal C^{k-1}$
on $U$. 

We  claim that the Lagrangians $\mathcal W (\hat h_q)$  of $L$-Jacobi fields  along $\gamma ^{\hat h_q}$ depend in a $\mathcal C^{k-1}$ way on the base point $q$ in a neighborhood of $x$.  Once the claim is proven, we apply it to some basis 
 $\hat h^i$  of the space $\mathcal B(L)$ of all basic normal fields  along $L$,  such that any $\hat h^i$ is typical and $\hat h^i _x$ is sufficiently close to $h$, hence contained in $\mathcal U$.    The statement that  the maps $q\to \mathcal W(\hat h_q^i )$ are $\mathcal C^{k-1}$ would  imply that $L$ is $\mathcal C^{k+1}$ in a neighborhood of $x$ by 
Corollary \ref{cor:  contdepend}.  

It remains to prove the claim.   By assumption, $\mathrm{rk} (\hat h_q)$ is a constant $b$ in a neighborhood $O$ of $x$ in $L$.   Hence, for $q\in O$ the space of 
 focal-generated Jacobi fields,  $^o\mathcal W(\hat h_q)  \subset {\mathcal W}(\hat h_q)  $  has dimension $n-1-b$.


We  fix focalizing Jacobi fields $J_i \in \mathcal W^{t_i} (\hat h_x)$,  for $i=1,...,n-1-b$,
which generate $\tilde {\mathcal W}(\hat h_x)$.   By the inductive assumption and  Proposition \ref{jac: family}, the families $\mathcal W^{t_i} (\hat h_q)$  of Jacobi fields focalizing at the time $t_i$, depend in a $\mathcal C^{k-1}$-way on the point $q\in L$.

In particular, we can extend the Jacobi fields $J_i$ to a $\mathcal C^{k-1}$ family of Jacobi fields $q\to J_i^q\in  \mathcal W^{t_i} (\hat h_q)$ on a neighborhood $O_1\subset O$ of $x$ in $L$.   

By continuity, the $(n-b-1)$ Jacobi fields $J_i^q$ are linearly independent, for $q$ sufficiently close to $x$.  By the dimensional considerations above, these Jacobi fields constitute a basis of $^o{\mathcal W} (\hat h_q)$, for all $q\in O_1$.  Thus, the map $q\to  {^o\mathcal W}(\hat h_q)$ is
$\mathcal C^{k-1}$ on $O_1$.

Consider for $q\in O$ the space of initial values of focal-generated fields:
$$F^q:=\{J(0)   \;\;  : \;\;  J\in     {^o\mathcal W}(\hat h_q) \}  \subset V^q\;.$$
By assumption on the rank, $\dim (F^q)=n-m-b$, for all $q\in O$.  Since this dimension is independent on $q$, the map $q\to F^q$ is $\mathcal C^{k-1}$ on $O_1$. 

Then also the orthogonal complement of $F^q$ in $V^q$ depends in a $\mathcal C^{k-1}$
on  $q\in O_1$. Since this orthogonal complement is exactly the set of initial values of elements $\mathcal W^{par} (\hat h_q)$, we deduce that  the space of parallel fields $\mathcal W^{par} (\hat h_q)$ depends in a $\mathcal C^{k-1}$ way on $q\in O_1$.
  Therefore, 
$$q\to \mathcal W(\hat h_q)={^o\mathcal W}(\hat h_q)  \oplus  \mathcal W^{par} (\hat h_q)$$
is $\mathcal C^{k-1}$ as well.
This finishes the proof of the claim, of the inductive step and of the theorem.
\end{proof}

A combination of Theorem \ref{thm: c2neu}  and Theorem \ref{thm: openc}  now finishes the proof of Theorem \ref{thm: main}.

If there are no parallel normal fields along any geodesic in $M$, then  the rank of any normal non-zero vector $h\in \mathcal H_{reg}$ is $0$.  Hence, the set $\mathcal U$ defined in Theorem \ref{thm: openc} coincides with $\mathcal H_{reg}$ and the submetry
$P$ is smooth on  the regular part of $M$, hence on all of $M$. 

In particular, this applies if the curvature of $M$ is positiv, proving the first  part of Theorem \ref{thm: maincompact}:
\begin{cor}
Let $M$ be complete and positively curved. Then any transnormal submetry $P:M\to Y$ is smooth.
\end{cor}


\newpage

\centerline{\bf V.  Dual leaves and smoothness}

\section{Dual leaves} \label{sec: dual}

\subsection{Singular  foliations} \label{subsec: foliation}
%
Let $\mathcal G$ be a family 
of $\mathcal C^k$-flows $\Phi : M\times \R\to  M$, for some $k\geq 1$.  As usual, we write $\Phi _t(x):=\Phi (x,t)$.

   Denote by $\Psi=\Psi ^{\mathcal G}$ the group of all $\mathcal C^k$-diffeomorphisms $\phi:M\to M$, which can be  represented as 
\begin{equation} \label{eq: repr}
\phi =\Phi ^1_{t_1} \circ ...\circ \Phi ^l_{t_l}\;,
\end{equation}
for some $l$, some  $t_i\in \mathbb R$ and some flows $\Phi ^i \in \mathcal G$. We recall the following result from  \cite[Theorem 1]{Stefan}:

\begin{thm} \label{thm: stefan}
 In the situation described above, each orbit $\Psi\cdot x$ is an injectively immersed $\mathcal C^k$-manifold  in $M$.  
\end{thm}

Indeed, \cite[Theorem 1]{Stefan} also proves that the decomposition of $M$ in $\Psi$-orbits is a \emph{singular $\mathcal C^k$-foliation}:  
For any $x\in M$ there exists a $\mathcal C^k$ diffeomorphism $F: O_1\times O_2 \to U$ with the following properties:
\begin{itemize}
\item $U$ is a neighborhood of $x$ in $M$,
\item  $O_i$ are open in $\mathbb R^{k_i}$,
\item $k_1=\dim (\Psi \cdot x)$,
\item For every $y\in U$   there exists a subset $W_y$ of $O_2$
such that $$F(O_1\times W_y) =\Psi \cdot y \cap U\;.$$
\end{itemize}

We will use the word \emph{leaves} to denote the $\Psi$-orbits under the assumptions of Theorem \ref{thm: stefan}.  Given such a leaf $\mathcal L$ we call its \emph{intrinsic completion} $\tilde {\mathcal L}$   the set of points $z$  in $M$, such that
there exists a Lipschitz curve $\gamma:[0,1]\to M$, with $\gamma (1)=z$ and $\gamma (t)\in \mathcal L$ for all $t<1$.


\begin{lem} \label{lem: leafclosure}
Consider the decomposition of a manifold $M$ into leaves $\mathcal L_x:=\Psi \cdot x$ as in Theorem \ref{thm: stefan}.  Then the closure $\bar {\mathcal L}$ and the completion $\tilde {\mathcal L}\subset \bar {\mathcal L}$ of any leaf $\mathcal L$ are unions of leaves of dimension at most $\dim (\mathcal L)$.    
  All leaves in 
$\tilde {\mathcal L}\setminus  \mathcal L$ have dimension smaller than $\dim (\mathcal L)$.
\end{lem}

\begin{proof}
Since the leaves are orbits of a group of diffeomorphisms, $\bar {\mathcal L}$ and $\tilde  {\mathcal L}$ 
are $\Psi$-invariant, hence they are union of dual leaves.  The dimension of leaves of a singular foliation is semi-continuous,  \cite[p. 700]{Stefan}, thus, all leaves contained in $\bar {\mathcal L}$ have dimension at most $\dim (\mathcal L)$.

Given a  leaf $\mathcal L'\subset \tilde {\mathcal L} \setminus \mathcal L$, and a point $x\in \mathcal L'$ we consider a trivializing chart $U$ at $x$, as described above. If $\dim (\mathcal L')=\dim (\mathcal L)$ then $\mathcal L\cap U$ is a countable union  of parallel translates of $\mathcal L' $.  Hence, a curve  contained  in $\mathcal L$ cannot converge to a point $x$, in contradiction to $\mathcal L'\subset \tilde {\mathcal L}$.
\end{proof}

We say that a leaf $\mathcal L$ is \emph{intrinsically complete} if $\mathcal L=\tilde {\mathcal L}$.  Due to  Lemma 
\ref{lem: leafclosure}, all leaves of minimial dimension are intrinsically complete.
If the manifold $M$ is complete, a leaf $\mathcal L$ is intriniscally complete if and only if $\mathcal L$ is complete as metric space, when equipped with its intrinsic metric.

\subsection{Singular Riemannian foliations} \label{sec: srf}
Let $M$ be a Riemannian manifold.  A $\mathcal C^k$ singular foliation on $M$ in the sense of Section \ref{subsec: foliation} is called a  $\mathcal C^k$ \emph{singular Riemannian foliation} if any geodesic  $ \gamma$ which intersects one leaf $\mathcal L$  of the foliation orthogonally is orthogonal to all leaves  intersected by $\gamma$.  We refer to \cite{Molino}, \cite{Alexandrino-survey} for the theory of singular Riemannian foliations with smooth leaves. 

 For a  smooth singular Riemannian foliation  on a complete Riemannian manifold the leaves are pairwise equidistant and so are their closures.  The leaf closures are again leaves of a singular Riemannian foliation, \cite{ARMolino}, \cite{Al-RadI}, \cite{Alexandrino-survey}.  We will use the following much simpler result in the non-smooth setting.

\begin{prop} \label{prop: leafclos}
Let $M$  be a complete Riemannian manifold. For any $\mathcal C^1$ singular Riemannian foliation $\mathcal F$ on $M$ and  any pair of leaves $\mathcal L_1, \mathcal L_2$, the leaf closures $\mathcal L_1$ and $\mathcal L_2$ are  equidistant.

 The leaf closures of $\mathcal F$ are fibers of a transnormal submetry 
$$P_{\bar {\mathcal F}}:M\to M/\bar{\mathcal F}\,.$$
\end{prop}

\begin{proof}
Let $L$ be the closure of  $\mathcal L_1$ and let $f$  be the distance function $f=d_L=d_{\mathcal L_1}$.  
Assume that $\mathcal L_2$ is not contained in $L$ and let 
 $x\in \mathcal L_2  \setminus L$ be arbitrary. The distance from $L$ to $x$ is realized by
some geodesics from $x$ to some points $z_i \in L$.  For any such $z_i$, its leaf $\mathcal L_{z_i}$ is contained in $L$, Lemma \ref{lem: leafclosure}. By the first formula of variation, any shortest geodesic from $x$ to $z_i$ meets $\mathcal L_{z_i}$ orthogonal at $z_i$.

By the definition of a singular Riemannian foliation, any such geodesic meets $\mathcal L_2$ orthogonal at $x$.  By the first formula of variation, the differential $D_xf (v)$  vanishes, for any 
direction $v\in T_x M$ tangent to the manifold $\mathcal L_2$.    Thus, $f$ is constant along any $\mathcal C^1$ curve in $\mathcal L_2$.  Thus, $f$ is constant of $\mathcal L_2$ and then also on the closure of $\mathcal L_2$.    

Exchanging the roles of $\mathcal L_1$ and $\mathcal L_2$, we deduce that their leaf closures     
are equidistant.  
Thus,  the decomposition of $M$ into leaf closures  of $\mathcal F$ defines a submetry 
$P_{\mathcal F}:M\to Z$, whose fibers are exactly these leaf closures.

It remains to verify that any leaf closure $L:=\bar {\mathcal L_1}$ is a manifold without boundary.  Let $\Psi$ denote the group of $\mathcal C^1$ diffeomorphsms of $M$ defining $\mathcal F$.  Then, $L$ is $\Psi$-invariant by Lemma \ref{lem: leafclosure}. 
Moreover, for any $q\in L$, the leaf $\Psi \cdot q$ has the same closure $L$, due to equidistance. 

$L$ is a subset of positive reach in $M$, \cite[Theorem 1.1]{KL}. If $L$ is not a manifold without boundary, $L$ contains a subset  $K$, open in $L$ and homeomorphic to a manifold  with non-emoty boundary, \cite[Theorem 1.2]{L-note}.  By the domain invariance, the group $\Psi$ of homeomorphisms of $L$ cannot map a boundary point $q$ of $K$ to an interior point of $K$.  Thus, the leaf $\Psi \cdot q$ cannot be dense in   $L$, in contradiction to  the equidistance statement.

Hence, $L$ is a manifold without boundary and  $P_{\mathcal F}$ is transnormal.
\end{proof}


\subsection{Dual leaves}  Let again $P:M\to Y$ be a  local submetry. 
For any $x\in M$, we denote by $\mathcal L^{\sharp}_x$  the set of points which can be connected to $x$ by a   piecewise horizontal geodesic and call it the \emph{dual leaf} of $x$.
Clearly, $M$ is a disjoint union of dual leaves.  As in the smooth situation, \cite[Proposition 2.1]{Wilking} we can now prove: 

\begin{prop} \label{prop: dualck}
Let the local submetry $P:M\to Y$ be  $\mathcal C^k$, for some $k\geq 2$.   Then the dual 
leaves of $P$ are leaves of a singular $\mathcal C^{k-1}$-foliation. In particular, they are immersed $\mathcal C^{k-1}$ submanifolds.  
\end{prop}

\begin{proof}
We consider the set  of all  triples  $( O,  \Phi, \mathcal D)$  such that:

\begin{enumerate}
\item $O$ is an open subset of $M$.
\item $\Phi$ is a flow of $\mathcal C^{k-1}$-diffeomorphisms  of $M$, whose support is compact and  contained in $O$.
\item $\mathcal D$ is an open subset of some fiber $L$ of $P$.
\item The normal exponential map $\exp _{\mathcal D}$ of $\mathcal D$  sends some open subset $U \subset T^{\perp} \mathcal D$  diffeomorphically onto $O$.
\item The pull-back of $\Phi $ to $U$   via $\exp _{\mathcal D}$ preserves every single normal space $T^{\perp} _q \mathcal D \cap U =H^q\cap U$. 
\end{enumerate}

Consider the set $\mathcal G$ of all flows $\Phi $ (for all triples as above) and the corresponding group $\Psi =\Psi ^{\mathcal G}$ of $\mathcal C^{k-1}$-diffeomorphisms.
We claim that  the $\Psi ^G$-orbit of $x$ coincides with the dual leaf $\mathcal L^{\sharp} _x$.  

Indeed, let  $\Phi  \in \mathcal G$ be arbitrary and let $s\in \mathbb R$ be fixed. If $z\in M\setminus O$  then $\Phi _{s}(z)=z$.  If $z\in O$ then $z$ and $\Phi _s (z)$  are given as $\exp_q (h_1)$ and $\exp _q (h_2)$ for some $q\in \mathcal D$ and some $h_1,h_2\in H^q$, due to the properties  (4) and (5) above.
  Then  $z$ and $\Phi _s (z)$ are  connected with $q$ by horizontal geodesics, hence $z$ and $\Phi _s(z)$ are in the same dual leaf $\mathcal L^{\sharp} _q$.  

 By induction on 
the number $l$ of factors in the representation \eqref{eq: repr}, we deduce  $\phi (x)\in \mathcal L^{\sharp} _x$, for any diffeomorphism
$\phi\in \Psi^{\mathcal G}$. 

In order to verify the reverse inclusion $\mathcal L^{\sharp} _x \subset \Phi (x) $, consider any horizontal piecewise geodesic $\eta$ starting at $x$.   For any point $p$ on $\eta$ we find a small neighborhood $\mathcal D$ in the fiber $L=L^p$  and a small normal disc bundle $U \subset  T^{\perp} \mathcal D$, such that the normal exponential map of $\mathcal D$ sends $U$ diffeomorphically onto its open image $O$.  We may further assume that $O\cap \eta$ is a concatenation of two geodesics meeting at $p$.

$U$ is a $\mathcal C^{k-1}$ bundle over $\mathcal D$. Hence,  for any $h_1,h_2 \in U\cap H^p$, we find a  compactly supported $\mathcal C^{k-1}$-flow of $U$ which preserves every fiber $U\cap H^q$  and sends at the time $s=1$ the point $h_1$ to $h_2$.  
Sending this flow forward to $O$ via the normal exponential map, we deduce that 
$\exp _p(h_1)$ and $\exp_p (h_2)$ are contained in the same $\Psi$-orbit.   Hence $U\cap \eta$ is contained in one  $\Psi$-orbit.  Covering the geodesic $\eta$ by finitely many such open sets $U$, we deduce that $\eta \subset \Psi \cdot x$.

This  finishes the proof of the claim.  Now the proposition follows from
  \cite[Theorem 1]{Stefan}.
\end{proof}

\subsection{Normal spaces to dual leaves}
Let  $P:M\to Y$ be a  local submetry with $\mathcal C^2$-fibers.
For any $x\in M$ we denote by $\mathcal L^{\sharp}_x$ the  dual leaf of $P$ through $x$.
Due to Proposition \ref{prop: dualck},  $\mathcal L^{\sharp}_x$ is a $\mathcal C^1$ submanifold of $M$.  The normal space $T_x^{\perp}\mathcal L^{\sharp} _x$ will be denoted by $A^x$  and called the \emph{dual-normal} space at $x$.

All horizontal geodesics starting at $x$ are contained in $\mathcal L^{\sharp}_x$.  Therefore,  $H^x\subset  T_x\mathcal L^{\sharp} _x$, hence the dual-normal space $A^x$ is contained in $V^x$.

The next  observation  probably holds in  greater generality. We state and prove the version sufficient for the subsequent consequence


\begin{lem}  \label{lem: dualtang+}
Let $P:M\to Y$ be a  local submetry with $\mathcal C^2$-fibers and let $\gamma:I\to M$ be a typical horizontal geodesic contained in a dual leaf $\mathcal L$.
   Let $J$ be a normal Jacobi field
 along $\gamma$,  which is given as a variational field of a variation of horizontal geodesics.
Let $t_0\in I$ be arbitrary.  If $J(t_0)=0$ or if $\gamma (t_0)$ is regular and $J(t_0)$ is tangent to   $\mathcal L$  then $J(t)$ is tangent to $\mathcal L$, for all $t\in I$. 
\end{lem}

\begin{proof}
If $J(t_0)=0$, then $J$ can be written as the  variational field of a variation of horizontal geodesics $\gamma _s$ with $\gamma _s(t_0)=\gamma (t_0)$ for all $s$, as has been observed in the proof of Proposition \ref{jac: family}.   By definition, all these geodesics $\gamma _s$ are contained in $\mathcal L$, hence $J(t)$ is tangent to $\mathcal L$ for all $t$.

Let $\gamma (t_0)$ be regular and $J(t_0)$ be tangent to $\mathcal L$. 
   Consider a curve $\eta$ in $\mathcal L$ which starts in $\gamma (t_0)$ in the direction of $J(t_0)$.  Since the 
horizontal vectors define a $\mathcal C^1$-bundle over $M_{reg}$,  we find a $\mathcal C^1$ horizontal field along $\eta$ extending $\gamma'(t_0)$.  Exponentiating this vector field we obtain a variation of $\gamma$ through horizontal geodesics $\gamma _s$ with 
$\gamma _s(t_0)=\eta (s)$.  By definition, these geodesics $\gamma _s$  are contained in $\mathcal L$, hence the corresponding Jacobi field $\tilde J$ is everywhere tangent to $\mathcal L$.     By construction, $J(t_0)=\eta'(0)=\tilde J(t_0)$.  The difference $J-\tilde J$
is a Jacobi field along $\gamma$ which vanishes in $t_0$.  Moreover, $J-\tilde J$ is given a variational field of a variation by horizontal geodesics, see Section \ref{subsec: jacident}.
Subtracting a Jacobi field everywhere tanegnt to $\gamma$, we may assume that $J-\tilde J$ is a normal Jacobi field.   By the previous consideration, $J-\tilde J$ is then everywhere tangent to $\mathcal L$. Hence $J$ is  everywhere tangent to $\mathcal L$ as well.
\end{proof}

As a direct consequence we deduce:

\begin{lem} \label{lem: dualtangent}
Let $P:M\to Y$ be a  local submetry with $\mathcal C^2$-fibers. Let $L$ 
be a typical fiber and $x\in L$ be arbitrary.
Let $0\neq h\in H^x$ be a typical direction and  $\mathcal W(h)$ be the space of $L$-Jacobi fields along  $\gamma ^h$.   Let $^o\mathcal W (h)$ be the  focal-generated subspace of $\mathcal W(h)$. Then, the tangent space to the dual leaf $ T_x \mathcal L^{\sharp}_x$
contains 
$$\mathcal T :=\{J(0)  \;\;  :  \;\;  J\in {^o\mathcal W} (h) \}\,.$$ 
 Hence, $A^x$ is contained in the orthogonal complement of $\mathcal T$ in $V^x$.
\end{lem}

In non-negative curvature we obtain as a consequence:

\begin{cor}
Let $M$ be complete and non-negatively curved.  Let $P:M\to Y$ be a transnormal submetry,
let $x\in M$ be an arbitrary regular point. 
Then for any $v\in A^x \subset V^x$ and any non-zero horizontal vector $h\in H^x$ the following holds true:

The holonomy field $J_{h,v}$  along $\gamma ^h$ with $J_{h,v}(0)=v$ is parallel and $J_{h,v} (t)$ is contained in $A^{\gamma ^h(t)}$ for all $t\in \mathbb R$.
\end{cor}

\begin{proof}
Due to Theorem \ref{thm: c2neu}, the submetry $P$ is $\mathcal C^2$, hence the dual leaves are leaves of a $\mathcal C^1$ foliations and Lemma \ref{lem: dualtangent} applies.

Due to the $\mathcal C^2$ property and the density of the set of typical  fibers and typical directions, it suffices to prove the statement under the assumption that the fiber  $L$ of $x$ is a typical fiber and $\gamma =\gamma ^h$ is a typical horizontal geodesic.

We fix such $h$ and any $v\in A^x$.
Denote by $\mathcal J\subset \mathcal W(h)$ the space of all $L$-Jacobi fields, for which 
$J(0)\in T_x\mathcal L^{\sharp} _x$.  
 Appyling Lemma \ref{lem: dualtangent} and the non-negativity of the curvature, we deduce from Theorem \ref{thm: GAG} that $v$ is the starting vector of a parallel  Jacobi  $J^v\in \mathcal W(h)$.   Moreover, $J^v(t)$ is orthogonal to $J(t)$, for any $t$ and any $J\in \mathcal J$.

  By Lemma \ref{lem: dualtang+}, the vector $J(t)$ is tangent to $\mathcal L_x^{\sharp}$ for all $t$ and all $J\in \mathcal J$.    By dimensional reasons, for any $t$ which is not $\mathcal W(h)$-focal,  the space
$$\{J(t) \; \; ; \;\;  J\in \mathcal J\}$$
coincides with the normal space to $\gamma'(t)$ within the tangent space $T_{\gamma(t)} \mathcal L^{\sharp}_x$.   Thus, for any such $t$, the vector $J^v(t)$ is orthogonal to $\mathcal L^{\sharp}_x$.  

Since this applies to almost all $t$, it is true for all $t$, by continuity. Since the normal space
to $\mathcal L^{\sharp}_x$ is vertical at any point, $J^v(t)$ is vertical, for all $t$. Hence $J^v$ is a holonomy field, therefore $J^v=J_{h,v}$.
 \end{proof}

\section{Continuation of smoothness revisited}

\begin{thm} \label{thm: dualleaf}
Let $P:M\to Y$ be a transnormal local submetry. Let $\mathcal L^{\sharp}$ be a dual leaf of $P$ and let $x,z\in \mathcal L ^{\sharp}$.  If $P$ is $\mathcal C^2$ on $M$ and smooth around    $x$ then $P$ is smooth around $z$.
%
\end{thm}

\begin{proof}
We proceed by induction on the dimension $n=\dim (M)$. If  $n=1$ the statement is clear: any local submetry is smooth on $M$.

Let the statement be true in  dimensions smaller than $n \geq 2$. 
Let $x,z\in \mathcal L^{\sharp}$ be as in the formulation of the theorem.
Arguing by induction on the number of edges of a horizontal  piecewise-geodesic curve connecting $x$ and $z$, we may assume that $x$ and $z$ are connected by a horizontal geodesic.  Moving $x$ slightly, we may further assume that $x$ is a regular point.    
Assume that $P$ is not smooth around $z$. 

 Let $L$ be the fiber of $P$ through $z$. By  Proposition \ref{thm: ext}, we find  some $r>0$, so that $B_r(z)\cap L$ is smooth. We may  assume that $\bar B_{10r}(z)$ is compact and that $r$ is smaller than the injectivity radius of $P(z)$ in $Y$. 

For any point $p\in B_r(z)\setminus L$, let $p_0$ denote the projection of $p$ to $L$ and, for any $0<\varepsilon <r$ let  $p_{\varepsilon}$ be the point on the horizontal geodesic $p_0p$ with distance $\varepsilon$ to $ p_0$.   

Denote by $\mathcal N^{\varepsilon}$ the set of points in $B_r(z)$ with distance $\varepsilon $ to $L$.  Since $L$ and the normal exponential map of $L$ are smooth, the set 
 $\mathcal N^{\varepsilon}$ is a smooth submanifold of dimension $n-1$.
The restriction of $P$ to $\mathcal N^{\varepsilon}$ is a local submetry  onto the $\varepsilon$-distance sphere $S_{\varepsilon} $ around $y$ in $Y$.

The map $p\to p_{\varepsilon}$ is smooth (since it can be defined via the smooth normal exponential map of $L$), moreover, it sends  fibers of $P$ to fibers of $P$.
Thus, $P$ is smooth   at  $p$ if and only if 
$P$ is smooth at $p_{\varepsilon}$.  Moreover, it happens if and only if  
the local submetry  $P_{\varepsilon}:=P:  \mathcal N^{\varepsilon} \to S_{\varepsilon} $ is smooth in a neighborhood of  $p$, for some $r>\varepsilon>0$.

On the horizontal geodesic $xz$,  any point $\tilde  x$ with $0< d(\tilde x, z)< r$ is regular and, 
by Corollary \ref{cor: extregfib},  $P$ is smooth in a neighborhood of $\tilde x$. We may replace $x$ by $\tilde x$ and assume that $x\in B_r(z)$.

By assumption, we find points $p^j$ converging to $z$, such  $P$ is not smooth in a neighborhood of $p^j$.  Then $P$ is also non-smooth at  $p^j_{s}$,  for any fixed $0<s<r$.    Choosing a subsequence, we may assume that the points $p^j_s$  converge to a point $p$ with $d(p,z)=d(p,L)=s$.  By construction, $P$ is not smooth at $p$.

On the other hand, $P$ is smooth  at  $x_s$. We replace $x$ by $x_s$.

The projection $\Pi^L$ onto $L$ is Lipschitz open, when restricted to the fiber $L^x$ of $P$ through $x$, \cite[Proposition 7.2]{L}.  Since  $P$ is $\mathcal C^2$, the projection  $\Pi^L$ is $\mathcal C^1$. Hence, the restriction of $\Pi^L$ to $L^x$  is a submersion. 
  Thus, we find a $\mathcal C^2$-submanifold $x\in Q$ of $L^x$ such that the projection 
$\Pi^L:Q\to L$ is a $\mathcal C^1$-diffeomorphism onto a neighborhood of $z$ in $L$.

We may assume that $Q$ is contained in the set $U$, where $P$ is smooth.  Then we also may assume that $Q$ is smooth.  For any $\varepsilon$, the subset 
$$Q_{\varepsilon} :=\{ q_{\varepsilon} \;\;  : \;\;  q\in Q\} \subset \mathcal N^{\varepsilon}\;$$
consists of points at which the local submetry $P_{\varepsilon}:\mathcal N^{\varepsilon} \to S_{\varepsilon}$ is smooth.  On the other hand, $P_{\varepsilon}$ is not smooth at $p_{\varepsilon}$.

We will obtain a contradiction, by showing that the closure of the  dual leaf $\mathcal L_{\varepsilon}^{\sharp} \subset \mathcal N^{\varepsilon} $ through the point $p_{\varepsilon}$ intersects non-trivially the set $Q_{\varepsilon}$.  Here, we consider the dual leaf  $\mathcal L_{\varepsilon}^{\sharp} $ with respect to $P_{\varepsilon}: \mathcal N^{\varepsilon}\to S_{\varepsilon} $.  Once the non-triviality of the intersection is verified, we arrive at a contradiction:  Indeed,  by the inductive assumption the open set of points in whose neighborhood   $P_{\varepsilon}$ is smooth cannot intersect the dual leaf   $\mathcal L_{\varepsilon}^{\sharp}$  of $p_{\varepsilon}$, hence also the closure of this dual leaf.

It thus remains to verify that $Q_{\varepsilon}$ non-trivially intersects the closure of the dual leaf $\mathcal L_{\varepsilon}^{\sharp} $ through $p_{\varepsilon}$, for some sufficiently small $\varepsilon >0$.

We fix the following  smooth coordinates on $B_r(z)$.
We use  exponential coordinates in $L$ centered in $z$, to identify $B_r(z) \cap L$ (for sufficiently small $r$)  with the Euclidean ball in $\mathbb R^l=T_zL=V^z$.   Making $r$ smaller, if needed,  we trivialize the normal bundle of $L$, thus identifying $B_r(z)$ with an open subset of the trivial bundle $V^z\times H^z=\mathbb R^l\times  \mathbb R^k$.
 We use these coordinates to compare  tangent vectors at different points.


By construction, $Q_{\varepsilon}$ converges (uniformly) in the $\mathcal C^1$-sense to $L$. Thus, for any $\delta >0$, and all  sufficiently
 small $\varepsilon<\varepsilon _0 (\delta)$
\begin{equation}
d(T_{q_{\varepsilon}} Q_{\varepsilon}, T_{q_0}L) <\delta\,,
\end{equation}
for all $q\in Q$.  The distance above is measured in the Grassmanian.

Since horizontal vectors converge to horizontal vectors, we can make $\varepsilon$ smaller and assume  that for all  $e\in \mathcal N^{\varepsilon}$
and  any unit horizontal directon $h\in H^e$, the differential of the projection to $L$ is very small: 
\begin{equation} \label{eq: dep}
D_e \Pi^L (h)<\delta\,.
\end{equation}

We denote by $\mathcal N$ the subset $\mathbb R^l \times \mathbb S^{k-1}$ of $\mathbb R^l \times \mathbb R^k$ and  by $\Pi:\mathcal N\to \mathbb R^l$ the projection onto the first coordinate.  Using the coordinates introduced above and
  rescaling  by $\frac 1 {\varepsilon}$,  we  identify $\mathcal N_{\varepsilon} \subset M$ with an open subset of $\mathcal N$.  The subsets $\mathcal N_{\varepsilon}$ of $\mathcal N$ become larger as $\varepsilon $ gets smaller and the union of $\mathcal N_{\varepsilon}$  is all of $\mathcal N$.

The pull-backed metric from $M$ gives a smooth Riemannian metric $g_{\varepsilon}$ on $\mathcal N_{\varepsilon}$.  The metrics $g_{\varepsilon}$ converge locally uniformly in the $\mathcal C^{\infty}$ sense to  the standard product metric on $\mathcal N$.

Each $\mathcal N_{\varepsilon}$ is a total space of a local submetry $P_{\varepsilon}$  
with all fibers $\mathcal C^2$.  Each $\mathcal N_{\varepsilon}$ contains a smooth submanifold $Q_{\varepsilon}$ which is projected by $\Pi$ diffeomorphically onto an open subset $\Pi (Q_{\varepsilon})$ of $\mathbb R^l$. Moreover,  $\Pi (Q_{\varepsilon})$  contains the $\frac 1 {\delta}$-ball around $0$ in $\mathbb R^l$.   Here and below $\delta$ goes to $0$ with $\varepsilon$.

 The tangent spaces at all points  of $Q_{\varepsilon}$ are  at most $\delta$-far from $\mathbb R^l$.    Making $\varepsilon$ smaller if needed and using the convergence of $g_{\varepsilon}$ to $g_0$, we may assume, in addition, the following: Any  geodesic $\gamma$  of length  less than $4$,  which starts orthogonally to $Q_{\varepsilon}$  and is parametrized by arclength satisfies  for all $t$ with $||(\Pi\circ\gamma) (t)) ||<10 $ the condition:
\begin{equation} \label{eq: 2delta}
||(\Pi \circ \gamma)'(t) ||< 2\delta\,.
\end{equation}

For any $P_{\varepsilon}$-horizontal unit vector $h$ at any  $e\in \mathcal N_{\varepsilon}$ we have by \eqref{eq: dep}:
\begin{equation} \label{eq: depi}
||D_e\Pi (h)|| <\delta\,.
\end{equation}

Finally, there is  a point  $p\in \Pi^{-1} (0)$, independent of $\varepsilon$.  As explained above, it is suffcient to verify that the closure $C_{\varepsilon}$ in $\mathcal N_{\varepsilon}$ of the $P_{\varepsilon}$-dual leaf  through $p$ intersects $Q_{\varepsilon}$, for some $\varepsilon >0$

We assume, on the contrary,  that  $C_{\varepsilon}$ does not intersect $Q_{\varepsilon}$ for all $\varepsilon >0$. Consider the distance function 
$$f_{\varepsilon}=d_{Q_{\varepsilon}} : C_{\varepsilon} \to \mathbb R\;.$$

 We have $f_{\varepsilon}(p)<\pi +\delta <4$, since the sphere $\Pi^{-1} (0)$ contains a point in $Q_{\varepsilon}$.
By assumption, $f_{\varepsilon}(e)>0$ for all $e\in C_{\varepsilon}$.

 Hence,  we find points $e_{\varepsilon}\in C_{\varepsilon}$ with $d(e_{\varepsilon},p)<10$, such that $f_{\varepsilon}$ does not decrease with velocity at least $\frac 1 2$ at $e_{\varepsilon}$ (hence the absolute gradient of $-f$ at $e_{\varepsilon}$ is less than $\frac 1 2$), see \cite[Lemma 4.1]{Lyt-open}.  

We fix such a point $e_{\varepsilon}  \in C_{\varepsilon}$.
The subset $C_{\varepsilon}$ is a closure of a dual leaf, hence it is a union of dual leaves. Thus, $C_{\varepsilon}$ contains the dual leaf $\mathcal L^{\sharp} _{e_\varepsilon}$ through  the point $e_{\varepsilon}$. 
Consider a shortest  geodesic from $e_{\varepsilon}$ to $Q_{\varepsilon}$ and denote by $w_{\varepsilon}$ the incoming  unit direction of this geodesic at $e_{\varepsilon}$.  Since this geodesic is orthogonal to $Q_{\varepsilon}$, we have  by \eqref{eq: 2delta}:
$$D_{e_\varepsilon}\Pi (w_{\varepsilon})<2\delta\;.$$

By  assumption on $e_{\varepsilon}$, 
 the tangent space $T_{e_{\varepsilon}}\mathcal L^{\sharp} _{e_\varepsilon}$ does not contain vectors 
$v$, which include  an angle less than $\frac \pi 3$ with $w_{\varepsilon}$  (with respect to the  standard product metric on $\mathcal N$).

Once $\varepsilon$ small enough,  any  unit tangent vector $u_{\varepsilon} \in T_{e_{\varepsilon} }\mathbb S^{k-1}$,  closest  to $w_{\varepsilon}$,  has the  follwoing property:

For any  $0\neq v\in  T_{e_{\varepsilon}}\mathcal L^{\sharp} _{\varepsilon}$  the angle between $v$ and $u_{\varepsilon}$  is larger than  $\frac\pi 4$.

Since the dual leaf is $\mathcal C^1$, we may  slightly move the point $e_{\varepsilon}$ in its dual leaf, so  that $e_{\varepsilon}$ becomes a typical point of $P$ and still, for some unit vector  $u_{\varepsilon} \in T_{e_{\varepsilon}} \mathbb S^{k-1}$ the above property holds true.

 Let $h_{\varepsilon}$ be any unit horizontal vector in $H^{e_{\varepsilon}}$ and consider the Lagrangian  $\mathcal W_{\varepsilon}:=   \mathcal W_{\varepsilon} (h_{\varepsilon})$ 
 of $L_{\varepsilon}$-Jacobi fields along $\gamma^{h_{\varepsilon}}$.  Here, $L_{\varepsilon}$ is the fiber of $P_{\varepsilon}$ through $e_{\varepsilon}$.

Denote by  $^o{\mathcal W_{\varepsilon}}$ the focal-generated subspace of the Lagrangian of normal Jacobi fields $\mathcal W_{\varepsilon}$.  
The vector $h_{\varepsilon}$ is horizontal hence belongs to  the tangent space $T_{e_{\varepsilon}}\mathcal L_{e_\varepsilon}^{\sharp}$.   

 Due to  Lemma \ref{lem: dualtangent},  the subspace 
$$^o\mathcal W_{\varepsilon} (0):=\{J(0) \; \; :  \;\; J \in {^o \mathcal W_{\varepsilon}} \}$$ 
belongs to the tangent space to the dual leaf at $e_{\varepsilon}$ as well.

We find a sequence of values of $\varepsilon \to 0$, such that the corresponding points $e_{\varepsilon}$ converge to some $e\in \mathcal N$ and the vectors $h_{\varepsilon}$ and $u_{\varepsilon}$ converge to unit vector $u,v\in T_{e} \mathcal N$.   Since $u_{\varepsilon}$, is tangent to $\mathbb S^{k-1}$ so is the vector $u$. Due to \eqref{eq: depi}, $h$ is tangent to $\mathbb S^{k-1}$ as well.

Taking another subsequence, we may assume that the Lagrangians $\mathcal W_{\varepsilon}$ converge to a Lagrangian $\mathcal W$ of normal Jacobi fields along $\gamma ^h$ in the non-negatively curved manifold $\mathcal N$.  Moreover, we may assume that the focal-generated subspaces $^o{\mathcal W}_{\varepsilon}$ converge to a subspace $\hat {\mathcal W}$ of $\mathcal W$.

Applying Lemma \ref{lem: indexsemi}, we deduce that $\hat {\mathcal W}$ contains the 
focal-generated subspace $^o {\mathcal W}$ of $\mathcal W$.  Due to Theorem \ref{thm: GAG},  
at the point $e=\gamma ^h(0)$, the orthogonal complement $K$ of $\hat {\mathcal W} (0)$ in $h^{\perp}$  is given by a space of starting values of parallel Jacobi fields in $\mathcal W$.  Due to the explicit structure of $\mathcal N$, $K$ must be a subspace of $T_e\mathbb R^l$.  Thus, $\hat {\mathcal W} (0)$ contains  all elements in $T_e\mathbb S^{k-1}$, orthogonal to $h$.

By construction, the space generated by $h$ and  $\hat {\mathcal W} (0)$ is contained in the limit of the tangent spaces  to dual leaves
$$ \mathbb R\cdot h +  \hat {\mathcal W} (0) \subset  \lim _{\varepsilon \to 0} T_{e_{\varepsilon} }  \mathcal L_{\varepsilon} ^{\sharp}\,.$$

Thus, by the defining property of the vectors $u_{\varepsilon}$, the unit vector point $u\in T_e\mathbb S^{k-1}$ is not contained in  $\mathbb R\cdot h +  \hat {\mathcal W} (0) $.
This contradiction finishes the proof of the Theorem.
\end{proof}

\section{Applications to non-negative curvature}  \label{sec: applic}
\subsection{General statement}  Let $M$ be  complete and  non-negatively curved and $P:M\to Y$ be  a transnormal submetry.  Due to Theorem \ref{thm: c2neu}, all fibers of $P$ are $\mathcal C^2$  and due to Theorem \ref{thm: main}, there is a dense open subset $O$ in $M$, such that $P$ is smooth on $O$.  There is a unique largest open subset $\mathcal O$ of $M$ with this property. 

 The intersection of any fiber with $\mathcal O$ is smooth.   Due to Theorem \ref{thm: dualleaf},  $\mathcal O$ is a union of dual leaves of $P$ and so is the complement $M\setminus \mathcal O$.    All dual leaves contained in $M\setminus \mathcal O$ are $\mathcal C^1$ submanifolds, whereas all dual leaves contained in $\mathcal O$ are smooth by Proposition \ref{prop: dualck}.

 We now observe that the proof of \cite[Theorem 2]{Wilking} applies literally to our situation 
and shows:

\begin{thm} \label{thm: dualcomplete} Let $P:M\to Y$ be a transnormal submetry, where $M$ is complete, smooth and non-negatively curved.  If all dual leaves of $P$  are complete  then the dual foliation is  an  (a priori non-smooth) singular Riemannian foliation of $M$.   
\end{thm}

Literally as in \cite[p. 1313]{Wilking}, we obtain the following result from Theorem \ref{thm: dualcomplete} and Rauch's rigidity:
\begin{cor} \label{cor: rauch}
Under the assumptions of Theorem \ref{thm: dualcomplete}, for any direction $v\in V^x$, which is orthogonal to the dual leaf $\mathcal L^{\sharp}$, the geodesic in the direction of $v$ is contained in the leaf $L$ of $P$ through $x$.  
\end{cor}


Using Theorem \ref{thm: dualcomplete} we can now verify:

\begin{thm} \label{thm: dualcomplete+} Let $P:M\to Y$ be a transnormal submetry, where $M$ is complete, smooth and non-negatively curved.  If all dual leaves of $P$  are complete then $P$ is smooth.  
\end{thm}

\begin{proof}
By Theorem \ref{thm: dualcomplete}, the dual foliation  is Riemannian.  By Propoposition \ref{prop: leafclos}, the leaf closures of $\mathcal F$ constitute the fibers of a transnormal submetry $P^{\sharp}:M\to Z$.    

 Since there exists an open set $\mathcal O$ consisting of smooth dual leaves, the closure of any such leaf is smooth again, \cite{Alexandrino-survey}.    Due to Proposition \ref{thm: ext}, all fibers of $P^{\sharp}$ are smooth.

Any two fibers of $P$ and of $P^{\sharp}$ intersect transversally and orthogonally, in the sense that the normal space of one fiber is contained in the tangent space of the other.  From Corollary \ref{cor: rauch} we deduce that the restriction $P^{\sharp} :L\to Z$  of $P^{\sharp}$ to any leaf $L$ of $P$  is a submetry.   Moreover,
on the intersection $L\cap \mathcal O$ of $L$ with the subset $\mathcal O\subset M$ on which $P$ is smooth,
all fibers of this restriction $P^{\sharp}:L\to Z$ are smooth (as transversal intersection of $L$ and smooth fibers of $P^{\sharp}$).

We claim that any fiber  $F$ of this restriction $P:L\to Z$ and its normal bundle $T^{\perp }F \cap TL$  in $L$ are smooth (seen as subbundles of $TM$).  Once this is verified, we deduce the smoothness of $L$ around $F$ by considering the normal exponential map of  $F$.

It remains to verify the claim.  The restriction  $P^{\sharp}:L\to Z$ is  smooth on $L\cap O$ (since the fibers of $P:L\to Z$ are the transversal intersections of $L$ with the corresponding fibers of $P^{\sharp}$). 

Hence  the normal bundle any such fiber $F\subset L\cap O$ within $L$ is smooth. Moreover, the basic normal fields of $F$ with respect to the restriction   $P^{\sharp}:L\to  Z$ are smooth.  Fix such a regular smooth fiber $F\subset L\cap O$.  Consider any horizontal geodesic $\gamma$ connecting $F$ to another fiber $F'$ and the corresponding holonomy map $Hol^{\gamma} :F\to F'$.
The holonomy map is given by exponentiation in $M$ of a smooth basic normal field, hence  $Hol ^{\gamma}$ is smooth.  Due to \cite[Proposition 7.2]{L}  the holonomy map is Lipschitz-open  (even if the ambient manifolf $L$ is not smooth). Hence $Hol^{\gamma}$  is a smooth submersion and $F'$ is a smooth submanifod of $M$.   Thus,  all fibers of $P^{\sharp}:L\to Z$ are smooth.

Since all fibers  of $P^{\sharp}:L\to Z$  are smooth submanifolds of $M$,   any basic holonomy field along any such fiber $F$ is smooth.  This already proves the normal bundle of $F$ within $L$  is smooth for any regular fiber $F$.

  For a non-regular fiber $F$, and $x\in F$ we choose any regular horizontal  vector  $h\in T_x L \cap T_x^{\perp} F$. We consider the regular fiber $F'$ through $z:=\exp (\varepsilon \cdot h)$. We now use that the projection $\Pi:F'\to F$ (which is exactly a holonomy map discussed above) is a smooth submersion.  We thus find a smooth submanifold $Q\subset F'$ which is projected by $\Pi$ diffeomorphically onto a neighborhood  $U$ of $x$ in $F$.   The preimage of $Q$ under the normal exponential map of $F$  defines  a smooth extension of $h$ to a normal vector field  on $U$ within $L$.   This proves smoothness of the normal bundle of $F$ within $L$ and finishes the proof of the Theorem.
\end{proof}

\subsection{Riemannian submersion}
We  now finish the proof of Theorem \ref{thm-sm1}.
Thus, let $M$ be a  compact Riemannian manifold and let $P:M\to Y$ be a Riemannian submersion.  By Theorem \ref{thm-sm3},  $Y$ is smooth.

If, in addition,  $M$ is non-negatively curved, then $P$ is $\mathcal C^2$ by Theorem \ref{thm: c2neu}.   The proof of \cite[Theorem 3]{Wilking} (given under the assumption that $P$ is smooth) applies literally 
and shows that  all leaves of the dual foliation of $P$ are complete.  Hence,  by Theorem \ref{thm: dualcomplete+},  $P$ is smooth on all of $M$.  This finishes the proof of Theorem \ref{thm-sm1}.

\subsection{Symmetric spaces}
Let now $M$ be a non-negatively curved symmetric space and let $P:M\to Y$ be an arbitrary transnormal submetry.   Then $P$ is $\mathcal C^2$ by Theorem \ref{thm: c2neu}.  Now, the arguments of \cite{Llohann} apply without changes to the present situation and show that the dual leaves of $P$ are complete. In fact, by  \cite{Llohann},  any dual leaf is a direct factor of $M$ in this case.  Thus, Theorem \ref{thm: dualcomplete+}  shows that all fibers of $P$ are smooth.  This proves the only remaining case stated in Theorem \ref{thm: maincompact} and finishes the proof of this theorem.





\bibliographystyle{alpha}
\bibliography{smooth}

\end{document}